\documentclass[11pt,reqno,a4]{amsart}

\usepackage[mathscr]{eucal}
\usepackage{enumerate}
\usepackage{epsfig,latexsym,amsfonts,amssymb,amsmath,amscd,
graphics, color, 
epic,epstopdf}
\newcommand{\bp}{\boldsymbol{p}}

\newcommand{\bw}{\boldsymbol{w}}
\newcommand{\K}{\ensuremath{\mathbb{K}}}

\newcommand{\Z}{\ensuremath{\mathbb{Z}}}

\newcommand{\R}{\ensuremath{\mathbb{R}}}
\newcommand{\C}{\ensuremath{\mathbb{C}}}

\newcommand{\PP}{\ensuremath{\mathbb{P}}}

\newcommand{\const}{\operatorname{const}}
\newcommand{\mult}
{\ensuremath{\mathrm{mult}}}

\newcommand{\Int}{\ensuremath{\mathrm{Int}}}

\newcommand{\conv}{\operatorname{Conv}}
\newcommand{\eps}{\ensuremath{\varepsilon}}

\newcommand{\val}{\ensuremath{\mathrm{val}}}
\newcommand{\divi}{\ensuremath{\mathrm{div}}}

\newcommand{\trop}{\ensuremath{\mathrm{trop}}}
\newcommand{\floor}{\ensuremath{\mathrm{floor}}}

\newtheorem{lemma}{Lemma}[section]
\newtheorem{theorem}[lemma]{Theorem}

\newtheorem{proposition}[lemma]{Proposition}
\newtheorem{definition}[lemma]{Definition}
\theoremstyle{remark}
\newtheorem{remark}[lemma]{Remark}

\newtheorem{construction}[lemma]{Construction}
\newtheorem{example}[lemma]{Example}

\usepackage{tikz}

\newcommand{\quader}{
  \begin{tikzpicture}[z={(.3cm,.2cm)}, 
    line join=round, line cap=round 
    ]
    \definecolor{cof}{RGB}{219,144,71}
    \definecolor{pur}{RGB}{186,146,162}
    \definecolor{greeo}{RGB}{91,173,69}
    \definecolor{greet}{RGB}{52,111,72}

    \coordinate (A1) at (0,0,0);
    \coordinate (A2) at (0,1,0);
    \coordinate (A3) at (2.5,1,0);
    \coordinate (A4) at (2.5,0,0);
    \coordinate (B1) at (0,0,2);
    \coordinate (B2) at (0,1,2);
    \coordinate (B3) at (2.5,1,2);
    \coordinate (B4) at (2.5,0,2);

    \draw[fill=cof,opacity=0.6](A1) -- (A2) -- (A3) -- (A4) -- (A1) ;
    \draw[fill=pur,opacity=0.6](A2) -- (A3) -- (B3) -- (B2) -- (A2) ;
    \draw[fill=greeo,opacity=0.6] (A3) -- (B3) -- (B4) -- (A4) -- (A3);
    \draw[dashed] (A1) -- (B1) -- (B4)  (B1) -- (B2);
  \end{tikzpicture}
}
\newcommand{\toblerone}{
  \begin{tikzpicture}[z={(.3cm,.2cm)}, 
    line join=round, line cap=round 
    ]
    \definecolor{cof}{RGB}{219,144,71}
    \definecolor{pur}{RGB}{186,146,162}
    \definecolor{greeo}{RGB}{91,173,69}
    \definecolor{greet}{RGB}{52,111,72}

    \coordinate (A1) at (0,0,0);
    \coordinate (A2) at (0,1,1);
    \coordinate (A3) at (2.5,1,1);
    \coordinate (A4) at (2.5,0,0);
    \coordinate (B1) at (0,0,2);
    \coordinate (B4) at (2.5,0,2);

    \draw[fill=cof,opacity=0.6](A1) -- (A2) -- (A3) -- (A4) -- (A1) ;
    \draw[fill=greeo,opacity=0.6] (A3) -- (B4) -- (A4) -- (A3);
    \draw[dashed] (A1) -- (B1) -- (B4)  (B1) -- (A2);
  \end{tikzpicture}
}
\newcommand{\simplex}{
  \begin{tikzpicture}[z={(1.3cm,.2cm)}, 
    line join=round, line cap=round 
    ]
    \definecolor{cof}{RGB}{219,144,71}
    \definecolor{pur}{RGB}{186,146,162}
    \definecolor{greeo}{RGB}{91,173,69}
    \definecolor{greet}{RGB}{52,111,72}

    \coordinate (A1) at (0,0,0);
    \coordinate (A2) at (1.5,0,0);
    \coordinate (A3) at (0,1.5,0);
    \coordinate (A4) at (0,0,1.5);

    \draw[fill=cof,opacity=0.6](A1) -- (A2) -- (A3)  -- (A1) ;
    \draw[fill=greeo,opacity=0.6] (A2) -- (A3) -- (A4) -- (A2);
    \draw[dashed] (A4) -- (A1);
\end{tikzpicture}
}

\begin{document}
\title[Tropical floor plans and enumeration 
of multi-nodal surfaces]{Tropical floor plans and enumeration 
of complex and real multi-nodal surfaces}
\author {Hannah Markwig}
\address {Hannah Markwig, Eberhard Karls Universit\"at T\"ubingen, Fachbereich Mathematik, Auf der Morgenstelle 10, 72076 T\"ubingen, Germany }
\email {hannah@math.uni-tuebingen.de}

\author{Thomas Markwig}
\address{Thomas Markwig, Eberhard Karls Universit\"at T\"ubingen, Fachbereich Mathematik, Auf der Morgenstelle 10, 72076 T\"ubingen, Germany}
\email {keilen@math.uni-tuebingen.de}

\author{Kristin Shaw}
\address{Kristin Shaw, University of Oslo, Department of Mathematics, Postboks 1053 Blindern 0316 Oslo, Norway}
\email{krisshaw@math.uio.no}

\author{Eugenii Shustin}
\address{Eugenii Shustin, School of Mathematical Sciences, Tel Aviv University, Ramat Aviv, Tel Aviv 69978, Israel}
\email {shustin@post.tau.ac.il}

\begin{abstract}
The family
of complex projective surfaces in $\PP^3$ of degree $d$ having precisely $\delta$ nodes
as their only singularities has codimension $\delta$ in the linear system $|{\mathcal O}_{\PP^3}(d)|$ for sufficiently large $d$
and is of degree $N_{\delta,\C}^{\PP^3}(d)=(4(d-1)^3)^\delta/\delta!+O(d^{3\delta-3})$. In particular, $N_{\delta,\C}^{\PP^3}(d)$ is polynomial in $d$.

By means of tropical geometry, we explicitly describe $(4d^3)^\delta/\delta!+O(d^{3\delta-1})$ surfaces passing through a suitable generic configuration of
$n=\binom{d+3}{3}-\delta-1$ points in $\PP^3$. These surfaces are close to tropical limits which we characterize combinatorially, introducing the concept of floor plans for multinodal tropical surfaces. The concept of floor plans is similar to the well-known floor diagrams (a combinatorial tool for tropical curve counts): with it, we keep the combinatorial essentials of a multinodal tropical surface $S$ which are sufficient to reconstruct $S$.

In the real case, we estimate the range for possible numbers of real multi-nodal surfaces
satisfying point conditions. 
We show that, for a special configuration $\bw$ of real points, 
the number $N_{\delta,\R}^{\PP^3}(d,\bw)$ of real surfaces of degree $d$ having $\delta$ real nodes and passing through $\bw$ is bounded from below by 
$\left(\frac{3}{2}d^3\right)^\delta/\delta!
+O(d^{3\delta-1})$.

We prove analogous statements for counts of multinodal surfaces in $\PP^1\times \PP^2$ and $\PP^1\times \PP^1\times \PP^1$.

\end{abstract}
\maketitle


\section{Introduction}


For $\delta<4(d-4)$, the family of surfaces of degree $d$ in $\PP^3$ having $\delta$ nodes as their only singularities is smooth of dimension
$n=\binom{d+3}{3}-1-\delta$ \cite{DPW,ST}, so it is natural to state the enumerative problems:
\begin{enumerate}\item[(C)] What is the degree of this family? Equivalently, what is the number $N_{\delta,\C}^{\PP^3}(d)$ of
surfaces $S$ of degree $d$ and with $\delta$ nodes passing through $n$ points in general position?
\item[(R)] What is the maximal possible number $N_{\delta,\R}^{\PP^3}(d,\bw)$ of real $\delta$-nodal surfaces of degree $d$
passing through a configuration $\bw$ of $n$ real points in general position?
\end{enumerate}
For $\delta=1$, the answer to question (C) is well-known: $N_{\delta,\C}^{\PP^3}(d)=4(d-1)^3$. Concerning question
(R) for $\delta=1$, we have proved in \cite{MMS15} the existence of a point configuration $\bw$ giving the lower bound $N_{\delta,\R}^{\PP^3}(d,\bw)\geq \frac{3}{2}d^3+O(d^2)$. To obtain this bound, we have used a tropical approach involving a lattice path algorithm.
The use of integration with respect to the Euler characteristic works only for odd-dimensional hypersurfaces (see an example of plane curves in Section
\ref{sec-uni}).

For arbitrary fixed $\delta$ and $d\gg\delta$, it is known that the answer to question (C) is polynomial in $d$, and the first coefficients of the polynomial are:
\begin{equation}N_{\delta,\C}^{\PP^3}(d)=\frac{(4(d-1)^3)^\delta}{\delta!}+O(d^{3\delta-3})\ .
\label{emn1}\end{equation}
This results from the standard computation of the class of the incidence variety, while $O(d^{3(\delta-1)})$
designates parasitic terms generated by the diagonals. 
On the other hand, there has yet to be any general approach to 
question (R).
In this paper, we address question (R), giving a lower bound similar to the case of $\delta=1$.

Our approach to this question is via tropical geometry, where instead we count tropicalizations of $\delta$-nodal surfaces with an appropriate multiplicity reflecting the number of (complex or real) $\delta$-nodal surfaces of degree $d$ that interpolate the given points and tropicalize to a particular tropical surface in the count.
In the complex case, we let
\begin{equation} N_{\delta,\C}^{\PP^3,\trop}(d)= \sum_S \mult_{\C}(S),\label{eq-corres}\end{equation}
where the sum goes over all tropicalizations $S$ of $\delta$-nodal surfaces of degree $d$ satisfying the point conditions, and $\mult_{\C}(S)$ accounts for the number of such surfaces which tropicalize to $S$. Then, the correspondence
$$N_{\delta,\C}^{\PP^3}(d)=N_{\delta,\C}^{\PP^3,\trop}(d)$$
holds by definition. Here, we work over the field of complex Puiseux series to employ tropicalization.

Unfortunately, the combinatorics of the tropicalizations $S$ of $\delta$-nodal surfaces can be quite intricate, and the task to enumerate them precisely is difficult.
Similar to how floor diagrams for tropical curves can be used to manage counts of tropical curves \cite{BM08, FM09}, we introduce floor plans for tropical surfaces suitable for our counting problem.
They give us a tool to combinatorially characterize a subset of the tropical surfaces $S$ contributing to the above sum, for a special choice of point configuration $\bw$ (containing $n=\binom{d+3}{3}-\delta-1$ points). Viewed asymptotically, the subset in question is large, and suffices to reproduce the leading coefficient when we count floor plans. One can think about the subset as the set of tropical surfaces for which the tropicalizations of the $\delta$ nodes are sufficiently far apart.

More precisely, we define the number of complex floor plans of degree $d$ with $\delta$ nodes to be
$$ N_{\delta,\C}^{\PP^3,\floor}(d)= \sum_F \mult_{\C}(F),$$
where the sum goes over all floor plans $F$ (see Definition \ref{def-floorplan}), each floor plan corresponds to a unique tropical surface $S$ contributing to the sum in (\ref{eq-corres}), and $\mult_{\C}(F)=\mult_{\C}(S)$.

Since the tropical surfaces arising from floor plans form only a subset of the surfaces $S$ contributing to the sum in (\ref{eq-corres}), we have
\begin{equation} N_{\delta,\C}^{\PP^3,\floor}(d)\leq N_{\delta,\C}^{\PP^3,\trop}(d)\label{eq-floortrop}.\end{equation}

Our first main result (see Theorem \ref{thm-floorplanasymptoticP3}) considers the asymptotics of the number $ N_{\delta,\C}^{\PP^3,\floor}(d)$: we have $N_{\delta,\C}^{\PP^3,\floor}(d)= (4d^3)^\delta/\delta!+O(d^{3\delta-1})$.
Thus, counting floor plans, we can recover the leading coefficient of the polynomial $N_{\delta,\C}^{\PP^3,\trop}(d)=N_{\delta,\C}^{\PP^3}(d)$, in spite of the fact that (\ref{eq-floortrop}) gives an inequality only.

In order to count floor plans, we have to determine the multiplicities $\mult_{\C}(S_j)$ of the corresponding tropical surfaces $S_j$. We use patchworking techniques to do this, in particular we have to solve local lifting equations involving the initials of the coefficients of a surface tropicalizing to $S$, and its singular points \cite{IMS09,MMS15}.

Having the combinatorial tool of floor plans at hand, we turn our attention to the case of real multinodal surfaces.
Pick a real point configuration $\bw$ (i.e.\ over the field of real Puiseux series) and let
\begin{equation} N_{\delta,\R}^{\PP^3,\trop}(d,\bw)= \sum_S \mult_{\R,\bw}(S),\label{eq-corresreal}\end{equation}
where now the sum goes over all tropicalizations of real $\delta$-nodal surfaces of degree $d$ passing through $\bw$, and $\mult_{\R,\bw}(S)$ accounts for the number of such surfaces which tropicalize to $S$. Again, the correspondence
$$N_{\delta,\R}^{\PP^3}(d,\bw)=N_{\delta,\R}^{\PP^3,\trop}(d,\bw)$$
holds by definition.

We adapt the count of floor plans to the real situation by setting $ N_{\delta,\R,s}^{\PP^3,\floor}(d)$ to be the number of real floor plans of degree $d$ with $\delta$ nodes,
$$ N_{\delta,\R,s}^{\PP^3,\floor}(d)= \sum_F \mult_{\R,s}(F),$$
where the sum goes over the same floor plans $F$ as above, each floor plan corresponds to a unique tropical surface $S$ in (\ref{eq-corresreal}), and $\mult_{\R,s}(F)=\mult_{\R,\bw}(S)$. Here, $s$ stands for a vector consisting of three signs for each point $p_i$ in $\bw$. We choose $s=((+)^3)^n$. To relate the count of floor plans to the count of tropical surfaces, as before we pick $\bw$ in horizontally stretched position, so it is only the sign vector $s$ which governs different lifting behaviour.

We determine $\mult_{\R,s}(F)=\mult_{\R,\bw}(S)$ with the same patchworking techniques as before, only now we search for  real solutions for the initials of the coefficients of a surface tropicalizing to $S$. From the lifting equations we set up, it turns out that along with the coefficients, the singular points in question are real. For some tropical surfaces, it cannot be decided whether the lifting equations have real or imaginary solutions. As a consequence, we disregard the floor plans corresponding to those in our count of floor plans, resp.\ we set $\mult_{\R,s}(F)=0$.

Thus, in the real setting, the subset of tropical surfaces $S$ arising from our floor plans 
which contribute to our count is  smaller than in the complex case: in addition to tropical surfaces for which the singularities are not sufficiently far apart, we exclude tropical surfaces for which the real multiplicity is unknown.
Nevertheless, we have
\begin{equation} N_{\delta,\R,s}^{\PP^3,\floor}(d)\leq N_{\delta,\R,\bw}^{\PP^3,\trop}(d,\bw)\label{eq-floortropreal}\end{equation}
for our special choice of $\bw$, and where $s=((+)^3)^n$ denotes the vector of signs of the points $\bw$.

Our second main result (see Theorem \ref{thm-realfloorplanasymptoticP3positive}) considers the asymptotics of the number $ N_{\delta,\R,s}^{\PP^3,\floor}(d)$: we have
$ N_{\delta,\R,s}^{\PP^3,\floor}(d) =\left(\frac{3}{2}d^3\right)^\delta/\delta!+O(d^{3\delta-1})$. Here, all the signs in $s$ are chosen positive.

Combining this with the inequality (\ref{eq-floortropreal}) and the correspondence (\ref{eq-corresreal}), we obtain:

\begin{theorem}\label{thm-realsurfacesconstructive}
For any given $\delta\ge1$ and any $d\gg\delta$ there exists a configuration $\bw$ of $n$ real points
in $\PP^3$ such that there are at least
$\frac{1}{\delta!}(\frac{3}{2} d^3)^\delta+O(d^{3\delta-1})$ real $\delta$-nodal surfaces of degree $d$ passing through $\bw$, whose singular points are all real.
\end{theorem}

The fact that the complex count of floor plans produces the expected leading coefficient can be viewed as evidence that this estimate is reasonable, in any case we obtain the same leading exponent as for the complex count of surfaces in $\PP^3$, which is the highest possible.

We establish analogous results for asymptotic counts of surfaces in $\PP^1\times \PP^2$ and $\PP^1\times \PP^1\times \PP^1$
(see Theorems \ref{thm-realfloorplanasymptoticP2P1} and \ref{thm-realfloorplanasymptoticP1P1P1}
in Section \ref{sec-asympt}).

The contents of the paper are organized as follows. In Section \ref{sec-uni}, we highlight some special constructions of pencils containing a large number of  real curves and hypersurfaces in the uninodal case. We begin the tropical approach in Section \ref{sec-floorplanecurves}, where we review the notion of  floor decomposition for curves and also outline a tropical proof, due to Beno\^it Bertrand, of the existence of  pencils of real curves containing the maximal number of real solutions in $\PP^2$ and $\PP^1 \times \PP^1$.  In Section \ref{nsec3}, we define  floor plans of  singular surfaces and define their real and complex multiplicities. We also describe in this section  how these multiplicities relate to the real and complex multiplicities of the corresponding singular tropical surfaces.
Finally in Section \ref{sec-asympt}, we asymptotically enumerate singular multinodal surfaces in $\PP^3, \PP^2 \times \PP^1$ and $\PP^1 \times \PP^1 \times \PP^1$ over $\C$ and for a specific point configuration over  $\R$.

\subsection{Acknowledgements}
The fourth author was supported by the Israeli Science Foundation grants no. 176/15 and 501/18 and by the Bauer-Neuman Chair in Real and Complex Geometry.
The first and second author acknowledge partial support by the DFG-collaborative research center TRR 195 (INST 248/235-1).  The research of the third author  is supported by the Bergen Research Foundation project ``Algebraic and topological cycles in complex and tropical geometry". We would like to thank Madeline Brandt, Beno\^it Bertrand, Alheydis Geiger, Dmitry Kerner, Viatcheslav Kharlamov and Uriel Sinichkin for useful discussions and important remarks. Part of this project was carried out during the 2018 program on \emph{Tropical Geometry, Amoebas and Polytopes} at the Institute Mittag-Leffler in Stockholm (Sweden).  The authors would like to thank the institute for hospitality, and for providing excellent working conditions.

\section{Enumeration of real uninodal curves resp.\ hypersurfaces: special constructions}\label{sec-uni}

Here, we discuss special pencils of real curves resp.\ hypersurfaces containing a relatively large number of real singular
curves resp.\ hypersurfaces (which sometimes is larger than what we get
via tropical geometry). However, we point out that none of these constructions extends to the multi-nodal case, while the
tropical construction of uninodal curves resp.\ hypersurfaces extends to the multi-nodal case in a natural way.

\subsection{Real uninodal curves}
An elementary approach to enumerate real singular curves in a generic pencil is based on integration with
respect to the Euler characteristic \cite{Vir88}.
Consider a pencil spanned by two real smooth curves of degree $d$
intersecting in $d^2$ distinct real points. The pencil defines a projection of $\R P^2_{d^2}$ (the real projective plane blown up at $d^2$ real points) onto $\R P^1$. Integration with respect to the Euler characteristic  yields
$$\chi(\R P^2_{d^2})=\int_{\R P^1}\chi(\R C_t)d\chi(t)\quad\Longrightarrow\quad n_+-n_-=d^2-1\ ,$$ where $n_+$, resp. $n_-$ is the number of real curves with a real hyperbolic, resp. elliptic node.
In general, one can guarantee only that at least $d^2-1$ singular curves in such a pencil are real among $3(d-1)^2$ complex singular curves.

However, one can do much better. The following statement has been suggested to the authors by V. Kharlamov.

\begin{lemma}\label{l-uni}
For any $d\ge2$, there exists a pencil of plane curves of degree $d$ that contains $3(d-1)^2$ real uninodal curves and no other singular curves.
\end{lemma}

\begin{proof}
Let $L_1,...,L_d$ and $L'_1,...,L'_d$ be two collections of real lines in general position in $\PP^2$, and let
$K_1,K_2,K_3\subset\R P^2$ be disjoint closed discs such that all $\frac{d(d-1)}{2}$ intersection points $L_i\cap L_j$ lie in
$\Int(K_1)$, all $\frac{d(d-1)}{2}$ intersection points $L'_i\cap L'_j$ lie in $\Int(K_2)$, and all $d^2$ intersection points $L_i\cap L'_j$ lie in
$\Int(K_3)$. Now consider the real pencil ${\mathcal P}=\{\lambda L_1\cdot...\cdot L_d+\mu L'_1\cdot...\cdot L'_d\ |\ (\lambda:\mu)\in\PP^1\}$.
The rational function on the plane $$\varphi=\frac{L'_1\cdot...\cdot L'_d}{L_1\cdot...\cdot L_d}$$
vanishes along the lines $L'_1,...,L'_d$ in $K_2$, and hence has at least one real critical point in each of the $\frac{(d-1)(d-2)}{2}$
components of $K_2\setminus(L'_1\cdot...\cdot L'_d)$ disjoint from $\partial K_2$. Similarly, $\varphi$ (or, equivalently, $\varphi^{-1}$) has
at least one real critical point in each of the $\frac{(d-1)(d-2)}{2}$ components of $K_1\setminus(L_1\cdot...\cdot L_d)$ disjoint from $\partial K_1$.
In $K_3$, we have $(d-1)^2$ quadrangular components of the complement
$K_3\setminus(L_1\cdot...\cdot L_d\cdot L'_1\cdot...\cdot L'_d)$ disjoint from $\partial K_3$,
and in each of them $\varphi$ has at least one real critical point:
indeed, $\varphi$ vanishes along two opposite sides and takes infinite value along the other opposite
sides of each quadrangular component, and hence has at least one saddle point inside. Then we take generic real curves $C$, $C'$ of degree $d$ sufficiently close
to $L_1\cdot...\cdot L_d$ and $L'_1\cdot...\cdot L'_d$, respectively, and such that the intersection points $L_i\cap L_j$ and $L'_i\cap L'_j$,
$1\le i<j\le d$, all turn into saddle points of the rational function $\widetilde\varphi=\frac{C'}{C}$ with different critical values. Then, in the
pencil $\widetilde{\mathcal P}=\{\lambda C+\mu C'\ |\ (\lambda:\mu)\in\PP^1\}$, we observe
$$2\cdot\frac{(d-1)(d-2)}{2}+(d-1)^2+2\cdot\frac{d(d-1)}{2}=3(d-1)^2$$ real singular curves.
\end{proof}

It is also possible to use tropical geometry to find a pencil of plane curves containing $3(d-1)^2$ real uninodal curves. This approach is suggested by B. Bertrand \cite{Bertrand} and is outlined after we introduce the notion of floor decomposed curves in Section \ref{sec-floorplanecurves}.

\subsection{Real uninodal hypersurfaces}
In the higher-dimensional case we demonstrate an elementary construction of a pencil of real hypersurfaces that contains
the number of real singular hypersirfaces asymptotically comparable with the degree of the discriminant.

\begin{lemma}\label{l-unih}
For any $n\ge 3$ and $d\ge2$, there exists a pencil of real hypersurfaces of degree $d$ in $\PP^n$ that contains at least $2(d-1)^n$ real singular hypersurfaces.
\end{lemma}

\begin{proof}
Take the real polynomial
$$F(x_1,...,x_n)=\sum_{i=1}^n\prod_{j=1}^d(x_i-\xi_{ij}),\quad\xi_{ij}\in\R\ .$$ For suitable values of $\xi_{ij}$,
it has $(d-1)^n$ real critical points in a small disc $K_1\subset\R^n$ with different critical values. 
Consider the polynomial
$$G(x_1,...,x_n)=F(x_1+t,...,x_n+t)\ ,$$ where $t\in \R$ will be specified later. We claim that, for $t\gg0$, the rational function
$\varphi=\frac{F}{G}$ has at least $(d-1)^n$ real critical points in the disc $K_1$ and at least $(d-1)^n$ real critical points in the shifted disc
$K_2=K_1-t(1,...,1)$. Indeed, in the disc $K_1$ the critical point system reads
\begin{equation}F_{x_i}=F\cdot\frac{G_{x_i}}{G},\quad i=1,...,n\ .
\label{e-uni}\end{equation} Note that the factors $\frac{G_{x_i}}{G}$ in the right-hand sides in all the equations are of order $O(t^{-1})$. Hence, for
a sufficiently large $t>0$, we obtain (at least) $(d-1)^n$ real  solutions to the system (\ref{e-uni}) in the disc $K_1$. Similarly, in
$K_2$ we get at least $(d-1)^n$ real critical points. \end{proof}

%

\section{Preliminaries for the tropical approach}
 We work over the field of complex Puiseux series
$\K=\bigcup_{m\ge1}\C\{t^{1/m}\}$, which is an algebraically closed field of characteristic zero with a non-Archimedean valuation denoted by $\val$.
Tropicalization, denoted by $\trop$, is coordinate-wise valuation.
The field $\K$ possesses the natural complex conjugation involution and contains the complete real subfield $\K_\R$ of real Puiseux series.

We assume that the reader is familiar with tropical hypersurfaces and their dual Newton subdivisions. When we speak about the area (resp.\ volume) of polygons (resp.\ polytopes) in a Newton subdivision, we mean normalized area (resp.\ volume).
A tropical surface in $\R^3$ is of degree $d$ if it is dual to the polytope $\conv\{(0,0,0),(0,0,d),(0,d,0),(d,0,0)\}$.
A tropical surface is of bidegree $(d,e)$ if it is dual to the polytope $\conv\{(0,0,0),(0,0,d),(0,d,0),(e,0,0), (e,0,d),(e,d,0)\}$.
A tropical surface is of tridegree $(d,e,f)$ if it is dual to the polytope \begin{equation}\conv\{(0,0,0),(0,d,0),(0,0,e),(0,d,e),(f,0,0),
(f,d,0),(f,0,e),(f,d,e)\}\label{nepolytope}\end{equation}
(see Figure \ref{fig-polytopes}).

 \begin{figure}
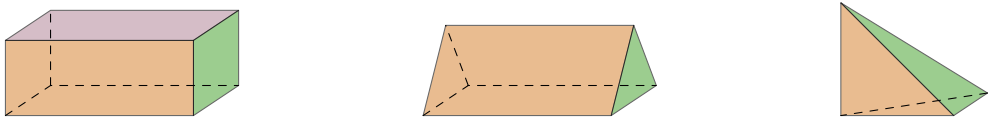

    \hfill\quader\hfill\toblerone\hfill\simplex\hspace*{1cm}
    \caption{Polytopes dual to the threefolds in which we count surfaces.}\label{fig-polytopes}
 \end{figure}

A \emph{lift} of a tropical surface $S$ is a surface in $\K^3$ tropicalizing to $S$. We call it a \emph{real lift} if it is defined over $\K_\R$.
A lift of a tropical surface of degree $d$ can canonically be compactified in $\PP^3$ and is then of degree $d$, a lift of a tropical surface of bidegree $(d,e)$ can be compactified in $\PP^1\times \PP^2$ to a surface of bidegree $(d,e)$ and a lift of a tropical surface of tridegree $(d,e,f)$ can be compactified in $\PP^1\times \PP^1\times \PP^1$ to a surface of tridegree $(d,e,f)$.

To determine the  lifting multiplicities (i.e.\ numbers of ($\delta$-nodal) surfaces of a certain (bi-,tri-) degree passing through $\bw$ which tropicalize to a given $S$), we rely on patchworking results from \cite{MMS15, IMS09}.


\section{Floor plans for singular tropical curves}\label{sec-floorplanecurves}
We start by presenting the analogue of our main combinatorial tool in one dimension lower --- as a tool to count plane tropical curves with one node. We can recover the well-known count of nodal complex curves satisfying point conditions in this way \cite{Mi03}.

Throughout this section, we pick a horizontally stretched point configuration in $\mathbb{R}^2$. For counts of curves of degree $d$ in $\PP^2$, we fix $n=\binom{d+2}{2}-2$ points. For counts of curves of bidegree $(d,e)$ in $\PP^1\times \PP^1$, we fix $n=(d+1)\cdot (e+1)-2$ points. We use coordinates $(x,y)$ for $\R^2$ and let $\pi_y:\R^2\rightarrow \R$ be the projection to $y$.
We denote the points by $q_1,\ldots, q_n$ and their projections by $Q_i=\pi_y(q_i)$.

Recall that all tropical plane curves passing through the horizontally stretched points are floor decomposed, i.e.\ all vertical lines appear in the dual Newton subdivision \cite{BM08, FM09}. The well-known tool of floor diagrams restricts the count of tropical plane curves to its combinatorial essence by shrinking each connected component of a tropical curve minus its horizontal edges to a single vertex.
Different from that, a floor plan encodes the positions of the horizontal edges under the projection $\pi_y$.

\begin{definition}\label{def-floorplandim2}
A \emph{floor plan for a tropical plane curve of degree $d$ with one node} consists of an index $j$ and a tuple $(D_d,\ldots,D_1)$ of tropical divisors on $\R$, where each $D_i$ is of degree $i$.

The divisor $D_i$ passes through the following points:
\begin{align*}
&\mbox{if }i>j:\;\;Q_{\binom{d+2}{2}-\binom{i+2}{2}+1},\ldots,Q_{\binom{d+2}{2}-\binom{i+1}{2}-1}\\
&\mbox{if }i=j:\;\;Q_{\binom{d+2}{2}-\binom{j+2}{2}+1},\ldots,Q_{\binom{d+2}{2}-\binom{j+1}{2}-2}\\
&\mbox{if }i<j:\;\;Q_{\binom{d+2}{2}-\binom{i+2}{2}},\ldots,Q_{\binom{d+2}{2}-\binom{i+1}{2}-2}.
\end{align*}

Furthermore, the divisor $D_j$ may have a point of weight $2$ if $j\not\in\{1,d\}$. If not, one of the points of $D_j$ has to align with a point of $D_{j-1}$ or of $D_{j+1}$.

Analogously, we define a \emph{floor plan for a tropical plane curve of bidegree $(d,e)$ with one node}. The differences are: the tuple of divisors is $(D_{d},\ldots,D_0)$, now each $D_i$ is of degree $e$, $D_j$ may have a point of weight $2$ if $j\not\in\{0,d\}$, and the divisor passes through the following points:
\begin{align*}
&\mbox{if }i>j:\;\;Q_{(d-i)\cdot(e+1)+1},\ldots,Q_{(d-i)(e+1)+e}\\
&\mbox{if }i=j:\;\;Q_{(d-j)\cdot(e+1)+1},\ldots,Q_{(d-j)(e+1)+e-1}\\
&\mbox{if }i<j:\;\;Q_{(d-i)\cdot(e+1)},\ldots,Q_{(d-i)(e+1)+e-1}.
\end{align*}

\end{definition}
Observe that in the case of degree $d$, for $i\neq j$, the divisor $D_i$ has $i$ points and has to meet $i$ points, while for $i=j$, $D_j$ has to meet only $j-1$ points.
In the same way, for degree $(d,e)$, $D_i$ meets $e$ points if $i\neq j$ and $D_j$ meets $e-1$ points.

\begin{figure}
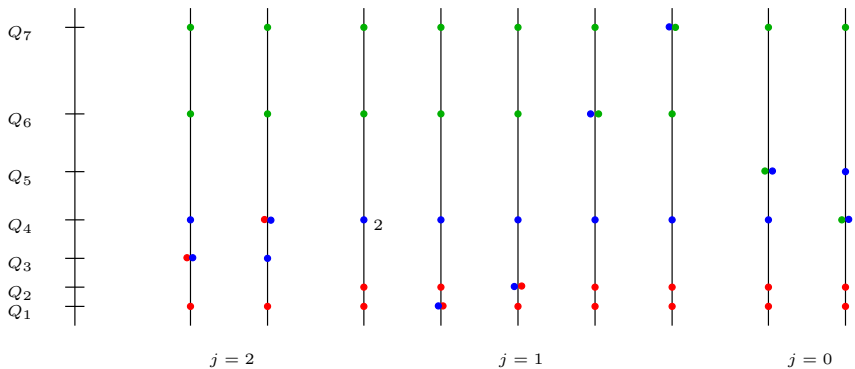\caption{Floor plans for $1$-nodal curves of bidegree $(2,2)$. The divisor $D_2$ is drawn with red dots, $D_1$ with blue and $D_0$ with green. On the left, the position of the points $Q_i$ is marked.}\label{fig-exdim1}
\end{figure}
\begin{construction}\label{const-curves}
Given a floor plan $F$ we construct a \emph{floor composed} tropical curve as in Section 3 \cite{BBLdM17}: we set $X_i=\R$ for all $i$, the effective divisors are the $D_i$. The rational functions $f_i$ are, since they must satisfy $\divi(f_i)=D_i-D_{i-1}$, determined up to a shift. We fix the shift such that the constructed floor meets the point condition whose projection lies between the points of $D_i$ and the points of $D_{i-1}$.
\end{construction}

\begin{example} Let $(d,e)=(2,2)$. We list all floor plans with one node for this degree. We have $n=7$ points in $\R$.
We have to consider tuples consisting of $3$ divisors of degree $2$ each. If $j=2$, $D_2$ passes through $Q_1$, $D_1$ through $Q_3$ and $Q_4$ and $D_0$ through $Q_6$ and $Q_7$. One of the points of $D_2$ has to be at $Q_3$, or $Q_4$.
If $j=1$, $D_2$ passes through $Q_1$ and $Q_2$, $D_1$ passes through $Q_4$ and $D_0$ through $Q_6$ and $Q_7$. Either $D_1$ has a point of weight $2$ at $Q_4$, or one of the points of $D_1$ is at $Q_1$, $Q_2$, $Q_6$ or $Q_7$. If $j=0$, $D_2$ passes through $Q_1$ and $Q_2$, $D_1$ through $Q_4$ and $Q_5$ and $D_0$ through $Q_7$. One of the points of $D_0$ is at $Q_4$ or $Q_5$.
Figure \ref{fig-exdim1} depicts all the possibilities.

Figure \ref{fig-exdim1const} shows an example of how a tropical plane curve is constructed as a floor composed curve from a floor plan.


\end{example}

\begin{figure}
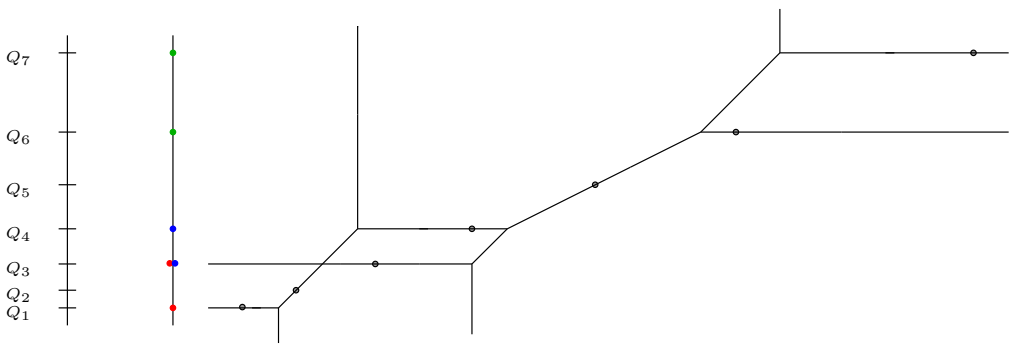\caption{How to construct a plane curve satisfying point conditions from a floor plan.}\label{fig-exdim1const}
\end{figure}

\begin{definition}\label{def-floorplandim2complexmult}
We define the complex multiplicity $\mult_\C(F)$ of a floor plan $F$ as above to be $4$ if it contains a divisor with a point of weight $2$, and $1$ otherwise.
\end{definition}

\begin{theorem}
For a floor plan $F$ for a tropical plane curve of degree $d$ (resp.\ of bidegree $(d,e)$) with one node, Construction \ref{const-curves} produces a unique element $C$ in the set of tropical plane curves of degree of degree $d$ (resp.\ of bidegree $(d,e)$) with one node meeting the horizontally stretched points $q_1,\ldots,q_n$.

Conversely, every tropical plane curve of degree $d$ (resp.\ of bidegree $(d,e)$) with one node meeting the horizontally stretched points $q_1,\ldots,q_n$ arises from a floor plan using Construction \ref{const-curves}.

Furthermore, the complex multiplicity of $F$ equals the complex multiplicity of $C$.

\end{theorem}
In particular, the count of floor plans with complex multiplicity equals the count of tropical curves with complex multiplicity, which by Mikhalkin's Correspondence Theorem \cite{Mi03} equals the number of complex $1$-nodal curves of degree $d$ (resp.\ bidegree $(d,e)$) satisfying point conditions.
\begin{proof}
By construction, we obtain from \ref{const-curves} a tropical curve satisfying the point conditions from a floor plan. If the floor plan has a divisor with a point of weight $2$, this tropical curve has a horizontal bounded edge of weight $2$. It is dual to an edge of weight $2$ in the dual subdivision, adjacent to two triangles of area $2$ each. If a non-fixed point aligns with another, the corresponding horizontal end passes through a floor, producing a parallelogram in the dual Newton subdivision. Besides those special features, the subdivision contains only triangles of area one, by the genericity of the point conditions and the fact that we construct ends of weight $1$ only. The curve has the right degree. Its multiplicity is $4$ if it has an edge of weight $2$ and $1$ otherwise.

Vice versa, given a tropical curve contributing to the count, we can produce a floor plan by projecting the horizontal edges between two floors to obtain the $D_i$. The index $j$ is obtained either by the position of the weight $2$ end, or by the position of a horizontal bounded edge which passes through a floor and does not meet a point between two adjacent floors. The projection thus either gives a divisor with a point of weight $2$, or a point which aligns with a point of the divisor corresponding to the neighbouring two floors.
\end{proof}

With the following proposition, we demonstrate how floor plans can be used to count efficiently.

\begin{proposition}\label{prop-complexcurvecount} The count of floor plans $F$ for a tropical plane curve of degree $d$ (resp.\ of bidegree $(d,e)$) with one node with complex multiplicity equals $3(d-1)^2$ (resp.\ $6de-4d-4e+4$).
\end{proposition}
\begin{proof}
We begin with the case of degree $d$.
If $j=d$, the point of $D_d$ which is not fixed by the $d-1$ conditions for $D_d$ can align with any point of $D_{d-1}$, so there are $d-1$ choices, each of multiplicity $1$. If $j=1$, the point can align with one of the two points of $D_2$, leading to $2$ choices. For $d>j>1$, we have $j-1$ choices for a point of weight $2$, leading to a summand of $4(j-1)$. The non-fixed point can also align with any of the $j+1$ points of $D_{j+1}$ and any of the $j-1$ points of $D_{j-1}$, leading to a summand of $2j$.
Altogether, we obtain \begin{align*}& d-1+2+\sum_{j=2}^{d-1}(4(j-1)+2j)= d+1+\sum_{j=2}^{d-1}(6j-4)\\&=
d+1+6\left(\frac{d(d-1)}{2}-1\right)-4(d-2)= 3d^2-6d+3=3(d-1)^2.\end{align*}

We now turn to the case of bidegree $(d,e)$.
If $j=e$, there are $d$ points with which the point of $D_d$ which is not fixed by the $d-1$ conditions for $D_d$ can align. The same holds if $j=0$.
For $e>j>0$, the non-fixed point can either form a point of weight $2$
with each of the $d-1$ fixed points --- each of these choices has
multiplicity $4$, so we have to sum $4(d-1)$ of those for each $j$ ---
or it can align with one of the $d$ points of $D_{j+1}$ or one of the
$d$ of $D_{j-1}$, leading to a summand of $2d$ for each such $j$.
Altogether, we obtain $d+(e-1)\cdot (4(d-1)+2d)+d=6de-4d-4e+4$.
\end{proof}

We now return to the tropical proof of the existence of a pencil of curves of degree $d$ containing $3(d-1)^2$ real nodal curves. This approach is due to Bertrand \cite{Bertrand}.
Since we will consider real curves, we must also choose a real configuration of points, meaning  for each point $q_i$ we must select the signs of its coordinates. To do this we fix  $\epsilon_i \in \mathbb{Z}_2^2$, where we think of $0, 1  \in  \mathbb{Z}_2$ as representing $+$ and $-$, respectively.

Bertrand introduces  the notion of signed marked floor diagrams for curves based on the notion of a real tropical plane curve from \cite[Definition 7.8]{Mi03} . Firstly, a real structure on a tropical plane curve $C$ is obtained by assigning
to each edge $e$ of $C$  an equivalence class $\mathcal{S}_e \in {\mathbb{Z}_2^2}/{w_e v_e}$ where $w_e$ is the weight of $e$ and $v_e$ is its primitive integer direction modulo $2$. These assignments are subject to compatibility conditions at the vertices.
For example, if $e$ is a horizontal edge  of weight $1$ of a floor decomposed curve $C$,  then a real structure on $C$ would assign an equivalence class $\mathcal{S}_e = \{(\sigma_1, \sigma_2) \sim (1+ \sigma_1 , \sigma_2)\}$, for some $(\sigma_1, \sigma_2) \in\mathbb{Z}_2^2$. If $e$ is a horizontal edge  of even weight  of $C$, then the equivalence class would consist of a single element $\mathcal{S}_e = \{(\sigma_1, \sigma_2)\}$.
An edge of weight $1$ of a floor of $C$ is assigned an equivalence class of the form $\mathcal{S}_e =\{(\sigma_1, \sigma_2) \sim (\sigma_1 + q ,\sigma_2 + 1)\}$ where $q = 0, 1$ depending on whether the first coordinate of $v_e$ is even or odd.

A real tropical curve $C$ passes through a  point configuration $q_1, \dots , q_k$ with sign vector $s = (s_1, \dots, s_k) \in (\mathbb{Z}_2^2)^k$  if the curve passes through the point configuration and for each point $q_i$ contained on an edge $e_i$  we have $s_i \in \mathcal{S}_{e_i}$.
Here we assume that the points $q_i$ are contained in the interior of the edges of $C$.


Bertrand also introduces  the real multiplicity of a signed marked tropical floor diagram, which here we relate back to floor decomposed curves.  
The non-trivial part of the definition occurs when the floor decomposed real tropical curve  has a horizontal  edge of weight $2$ and its complex multiplicity is $4$. The horizontal edge $e$ attaches to two consecutive  floors $X_i$ and  $X_{i+1}$ defined by divisors $D_{i} -  D_{i-1}$ and $D_{i+1}  -  D_i$, respectively.
Notice that if $e'$ and $e''$ are the edges of a floor adjacent to a horizontal  edge of $C$ of weight $2$, then $v_{e'} \equiv v_{e''} \mod 2$ and $\mathcal{S}_{e'} = \mathcal{S}_{e''}$.  

Suppose $X$ is a floor of the floor decomposed curve $C$ which is  adjacent to a horizontal  even weighted edge $\alpha$. Let $e'$ and $e''$ be  the two edges of the floor $X$
adjacent to $\alpha$.  The multiplicity assigned  to the pair $\alpha$ and $X$ is then
$$\mu_{\alpha, X} =
\begin{cases}
0 \text{ if } \mathcal{S}_ \alpha \not \subset  \mathcal{S}_{e'}  \\
2 \text{ if } \mathcal{S}_ \alpha  \subset \mathcal{S}_{e'}.
\end{cases}$$

\begin{definition}\label{def:realmultfloorcurves}

The $\mathbb{R}$-multiplicity of a marked real tropical curve $C$ passing through the point configuration $\{q_1, \dots , q_k\}$ with sign vector $s$
is
$$\mu_C = \prod_{X \text{ a floor of } C} \mu_X$$
where
$$\mu_X = \prod_{\alpha \text{ even edge adjacent to }  X} \mu_{\alpha, X} , $$
and $\mu_X = 1$ if there is no even weighted  edge adjacent to $X$.
\end{definition}

\begin{example}\label{ex:realweight2curve}
Figure \ref{fig:weight2curve} shows a portion of a real tropical  floor decomposed curve between the floors $X_5$ and $X_6$. The part of the curve is shown passing through a configuration of $6$ real points. The signs of these points, read from left to right, are fixed to be  $\epsilon, \epsilon, \epsilon' , \epsilon, \epsilon'  \epsilon'$, where
$\epsilon$ is any choice of pair of signs and $\epsilon' := \epsilon + (1, 0)$ if we identity $\{+, -\}$ with $\{0, 1\}$.

The real structure on the curve is given by prescribing an equivalence class of signs $\mathcal{S}_e$ on each edge. The horizontal edges of the curve all have $\mathcal{S}_e = \{ \epsilon, \epsilon'\}$, except for the edge $\alpha$ of weight $2$ which is only labelled with $\mathcal{S}_{\alpha} = \{\epsilon\}$.  Let $\epsilon'' = \epsilon + (0, 1)$ and $\epsilon''' = \epsilon + (1, 1)$. Then the labels next to an edge in the figure indicate the equivalence class $\mathcal{S}_e \subset \Z_2^{2}$ of the other edges of the curve.
By Definition \ref{def:realmultfloorcurves}, we have $\mu_{X_5, \alpha} = \mu_{X_6, \alpha} = 2$.

\end{example}

\begin{figure}\label{fig:weight2curve}
\includegraphics[scale= 0.7]{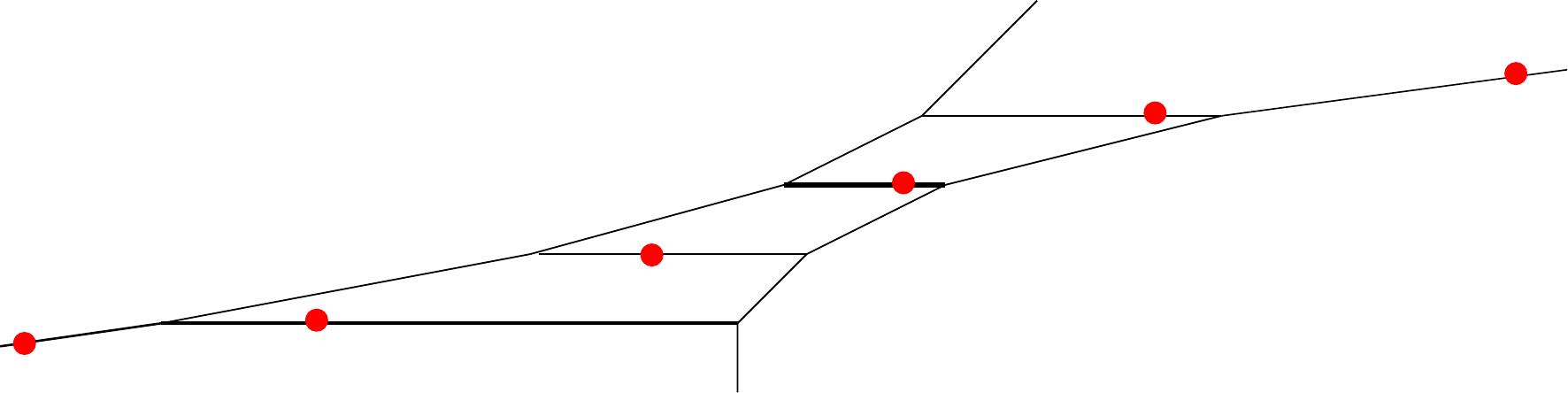}
\put(-395,10){$\{ \epsilon, \epsilon''\}$}
\put(-300,32){$\{ \epsilon', \epsilon''\}$}
\put(-240,45){$\{ \epsilon, \epsilon''\}$}
\put(-200,58){$\{ \epsilon, \epsilon''\}$}
\put(-170,80){$\{ \epsilon', \epsilon''\}$}
\put(-30,60){$\{ \epsilon', \epsilon'''\}$}
\put(-110,45){$\{ \epsilon, \epsilon'''\}$}
\put(-150,30){$\{ \epsilon, \epsilon'''\}$}
\put(-175,15){$\{ \epsilon', \epsilon'''\}$}
\put(-183,0){$\{ \epsilon, \epsilon'''\}$}
\put(-170,38){$2$}
\caption{A portion of a real tropical floor decomposed curve with a weight $2$ edge passing through a real point configuration from Example \ref{ex:realweight2curve}. The signs of the red marked points read from left to right are $\epsilon, \epsilon, \epsilon' , \epsilon, \epsilon'  \epsilon'$.
}
\end{figure}
\begin{theorem}\cite{Bertrand}\label{thm:maximalconfigtropicalcurves}
For every $d$ there exists a signed configuration of $ \frac{d(d+3)}{2} - 1$ points such that the number of uninodal marked real tropical curves passing through this point configuration counted with the multiplicities from Definition \ref{def:realmultfloorcurves} is equal to $3(d-1)^2$.
\end{theorem}
\begin{proof}
Firstly, by \cite{Bertrand} the multiplicity of a marked real tropical curve $C$ passing through the point configuration  $\{q_1, \dots , q_k\}$ with sign vector $s$ from Definition \ref{def:realmultfloorcurves}
corresponds to the number of real curves tropicalizing to $C$ and passing through a  configuration of real points which tropicalize to $\{q_1, \dots , q_k\}$. 

We choose a horizontally stretched point configuration. Notice that $$ \frac{d(d+1)}{2} - 1 + d  =  d + \sum_{i = 2}^d j,$$ so to specify the signs of the points we do it in ``blocks''. The first block of $d$ signs can be chosen arbitrarily. Then the signs are fixed in a block of size $d$, then $d-1$, $d-2$ and so on, until the final block of size two. 



Fix an $\epsilon \in (\Z_2)^2$ and let  $\epsilon' := \epsilon +(1, 0)$. The first sign of a block with $j$ points is $\epsilon$ considered in $(\Z_2)^2$ and the second sign of a block is the same as the first. Then the signs alternate between $\epsilon$ and $\epsilon'$ until the last sign of the block which is equal to the sign of the  $(j-1)$-th sign.
As an example, in degree $7$ a choice of signs of the points written in the order $Q_1, \dots, Q_{\frac{d(d+1)}{2} + d - 1 }$ is
$$\epsilon_1, \epsilon_2, \epsilon_3, \epsilon_4, \epsilon_5, \epsilon_6, \epsilon_7| \epsilon, \epsilon, \epsilon' , \epsilon, \epsilon'  \epsilon, \epsilon   |\epsilon, \epsilon , \epsilon', \epsilon , \epsilon', \epsilon' |  \epsilon , \epsilon, \epsilon' , \epsilon, \epsilon   | \epsilon ,\epsilon,  \epsilon', \epsilon' |   \epsilon,  \epsilon,  \epsilon | \epsilon, \epsilon.$$
The bars  ``$|$'' in the list are only used to delimit a block and can be removed.

Suppose $C$ is a uninodal tropical curve with no weight $2$ edges and passing through a point configuration. No matter what choice of sign vector for the points, it is possible to equip $C$ uniquely with signs so that the signed tropical curve passes through the signed point configuration \cite{Bertrand}. Therefore, if the tropical curve has no edges with even weight the real and complex tropical multiplicities are both equal to one.

 Otherwise, a uninodal tropical curve can have at most one weight $2$ edge. If it does have an edge of weight $2$, its complex multiplicity is equal to $4$.
When the tropical curve is floor decomposed the weight $2$ edge is a horizontal edge and we can suppose it is a weight $2$ edge joining the  floors $X_{j-1}$ and $X_{j}$
together with  another  $j - 3$ edges of weight one, for some $3 \leq j \leq d$.
Using the partition of the sign vector into blocks, the sign of the point marking the floor $X_{j}$ is $\epsilon$ and the sign of the point on the floor $X_{j-1}$ is
$\epsilon + (j,0)$. In other words, the first and last signs  in the $j$-th block are the signs of the marked point on  the floors $X_{j}$ and  $X_{j-1}$, respectively.

Denote the weight $2$  edge by $\alpha$ and suppose that it is marked by the  $k$-th sign of the block of size $j$, so that there are exactly $k-2$ edges of weight one that come before it. Notice that  $2 \leq k \leq j -1 $ since the first and last signs of the block mark the floor.   Therefore, we have
$\mathcal{S}_ {\alpha} = \{ \epsilon \} $ if $k$ is even and $\mathcal{S}_ \alpha = \{ \epsilon' \}$ if $k$ is odd.
Let $e$ and $e'$ be the two edges adjacent to $\alpha$ in the floor $X_{j}$. The  edge $a$ of the floor  $X_{j}$ marked by the point must have the  equivalence class of signs
$\mathcal{S}_{a} = \{\epsilon, \epsilon + ( j, 1)\}$, where again $ \epsilon + ( j, 1)$ is considered in $(\Z_2)^2$. Moreover, any weight one edge between the floors $X_{j}$ and $X_{j-1}$ has equivalence class of signs $\{\epsilon, \epsilon'\}$.   By the rule for propagating the equivalence classes around a trivalent vertex with only weight one edges we find that
$\mathcal{S}_{e} = \mathcal{S}_{e'} = \{\epsilon , \epsilon + ( j,
1)\}$  if $k$ is even and $\mathcal{S}_{e} = \mathcal{S}_{e'}
=\{\epsilon', \epsilon + ( j, 1)\}$  if $k$ is odd. Therefore, no
matter what the parity of $k$ is we have
$\mathcal{S}_{\alpha} \subset \mathcal{S}_{e} = \mathcal{S}_{e'} $. Thus $\mu_{\alpha, X_{j}} = 2$.

We can perform the analogous  check for the floor $X_{j-1}$.  We use the same convention as above and suppose that the weight $2$ edge is marked by the $k$-th sign of the block of size $j$. Then as above $\mathcal{S}_ \alpha = \{ \epsilon \} $ if $k$ is odd and $\mathcal{S}_ \alpha = \{ \epsilon' \}$ if $k$ is even. The edge $a$ of the floor  $X_{j-1}$ marked by the point is in direction $(j-1, 1)$ and is marked by $\epsilon + (j-1, 0)$, considered in $(\Z_2)^2$. Therefore,  the  equivalence class of signs is
$\mathcal{S}_{a} = \{\epsilon + (j-1, 0) , \epsilon + (0, 1)\}$. By the rule for propagating the equivalent classes of signs we find that  $\mathcal{S}_{a'} = \{\epsilon, \epsilon + (0, 1)\}$,
where $a'$ is the unbounded edge of the floor $X_{j-1}$ in direction $(0,1)$. Therefore, if $e$ and $e'$ are the edges of the floor adjacent to $\alpha$, we have
$\mathcal{S}_{e} = \mathcal{S}_{e'} = \{\epsilon, \epsilon + (0, 1)\} $ if $k$ is even  and
or $\mathcal{S}_{e} = \mathcal{S}_{e'} = \{\epsilon', \epsilon + (0,
1)\} $ if $k $ is odd. Once again, no matter what the parity of $k$ is we have
$\mathcal{S}_{\alpha} \subset \mathcal{S}_{e} = \mathcal{S}_{e'} $. Thus $\mu_{\alpha, X_{j-1}} = 2$.
Then the real  multiplicity of the
floor decomposed tropical  curve is equal to four which is the same as the complex multiplicity, therefore, the real count is equal to the complex count which is $3(d-1)^2$.
\end{proof}

Using floor decomposed real tropical curves it is also possible to find a real point configuration in $\PP^1\times\PP^1$ such that the number of real curves passing through the point configuration is maximal, i.e.~is equal to the complex number.

\begin{theorem}
For every bidegree $(d, e)$  there exists a real point configuration in $\PP^1\times\PP^1$ such that the number
of marked real tropical curves passing through this point configuration counted with the multiplicities from Definition \ref{def:realmultfloorcurves} is equal to $6de-4d-4e+4$.
\end{theorem}

\begin{proof}
Uninodal curves of bidegree $(d, e)$ in $\mathbb{P}^1 \times
\mathbb{P}^1$ are fixed by $(d+1)(e+1)- 2 = 2d + (d+1)(e-1)$
points. As in the proof of Theorem
\ref{thm:maximalconfigtropicalcurves}, we fix the sign vector of the
floor decomposed point configuration in blocks. We first choose a sign
$\epsilon$ and set $\epsilon' = \epsilon  + (1, 0)$. Then the first
block is of size $d$ and can be chosen arbitrarily. The next $e-1$
blocks are of size $d+1$ and are all the same. They begin with the sign
$\epsilon$, followed again by $\epsilon$ and then alternate until the
$d$-th sign. The last sign of the block is the same as the $d$-th
sign. Finally the last $d$ signs can be chosen arbitrarily. As an
example for $d= 4$ and $e = 3$ we have
$$\epsilon_1,  \epsilon_2, \epsilon_3, \epsilon_4 | \epsilon, \epsilon, \epsilon', \epsilon, \epsilon|    \epsilon, \epsilon, \epsilon', \epsilon, \epsilon| \epsilon_5, \epsilon_6,  \epsilon_7, \epsilon_8.$$
With the signs of the point configuration chosen we can follow the proof of Theorem \ref{thm:maximalconfigtropicalcurves} to establish that the real multiplicity of every real floor decomposed tropical curve is equal to the complex multiplicity, which proves the theorem.
\end{proof}

Bertrand's maximality argument does not generalize beyond the uninodal case. However, he does go further to consider the asymptotic count of real curves in $\mathbb{P}^2$ with a fixed number of nodes \cite{Bertrand}.
Using floor diagrams, Fomin and Mikhalkin  found that the number of complex curves of degree $d$  in $\mathbb{P}^2$ with $\delta$ nodes passing through $d(d+3)/2 -  \delta$ points  in general position is a polynomial  function in $d$ of degree $2 \delta$ \cite{FM}. By considering  real floor diagrams, Bertrand shows that for any choice of signs on the point configuration,  the number of real floor decomposed curves  with $\delta $ nodes counted with real multiplicity behaves as $\Theta(d^{2\delta})$. Furthermore, for a fixed number of nodes $\delta$, Bertrand exhibits a family of real point configurations such that, as $d$ tends to infinity the number of real  curves of degree $d$ passing through the point configuration has the same asymptotic as the complex case.

\section{Floor plans for singular tropical surfaces}\label{nsec3}

Throughout the rest of the paper, we pick a point configuration $\bw= (p_1,\ldots,p_n)$ consisting of $n$ points in $(\K^\ast)^3$, such that their tropicalizations (which we denote by $(q_1,\ldots,q_n)$) are distributed on a line in $\R^3$ with direction $(1,\eta,\eta^2)$, where $0<\eta\ll1$, and with growing distances.

We use $x$, $y$ and $z$ for the coordinates of $\R^3$ and $\pi_{yz}:\R^3\rightarrow \R^2$ for the projection to the $yz$-coordinates.
We let $Q_i=\pi_{yz}(q_i)$ for all $i=1,\ldots,n$.

As in Section \ref{sec-floorplanecurves}, a floor plan contains a tuple of effective divisors, now in $\R^2$ rather than $\R$, thus plane tropical curves.
We introduce  the notion of a string which will be used in our definition of floor plans:
Let $C$ be a plane tropical curve of degree $d$ (resp.\ of bidegree $(d,e)$) satisfying point conditions. Assume that the right or left corner of the Newton polygon appear in the subdivision (resp.\ the upper right or lower left), and that there are no point conditions on the two ends adjacent to the dual vertex. In this case we can prolong the adjacent bounded edge of direction $(1,0)$ or \ $(-1,-1)$ (resp.\ $(1,1)$ or $(-1,-1)$) arbitrarily, without losing the property of passing through the points. We call the union of the two ends a \emph{string} --- a right string if it is dual to the right corner, a left string if it is dual to the left corner.

\begin{definition}\label{def-nodegerm} Let $C$ be a plane tropical curve of degree $d$ (resp.\ of bidegree $(d,e)$) passing through $\binom{d+2}{2}-2$ (resp.\ $(d+1)(e+1)-2$) points in general position. A \emph{node germ} of $C$ is one of the following:
\begin{enumerate}
\item a vertex resp.\ midpoint of an edge of weight $2$, dual to a parallelogram resp.\ a pair of adjacent triangles of area $2$ each sharing an edge of weight $2$ in its Newton subdivision,
\item a horizontal or diagonal (resp.\ horizontal or vertical) end of weight $2$,
\item a right or left string.
\end{enumerate}
\end{definition}

\subsection{Surfaces in $\PP^3$}\label{surf-p3}
Here, $n=\binom{d+3}{3}-1-\delta$ with $0<\delta\le\frac{d}{2}$.

\begin{definition}\label{def-floorplan} A \emph{$\delta$-nodal floor plan $F$ of degree $d$} is a tuple $(C_d,\ldots,C_1)$ of plane tropical curves $C_i$ of degree $i$, together with a choice of indices $d\geq i_{\delta}\geq \ldots\geq i_1\geq 1$, 
such that $i_{j+1}>i_j+1$ for all $j$, satisfying the following properties:
\begin{enumerate}
\item The curve $C_i$ passes through the following points (where we
  set $i_0=0$ and $i_{\delta+1}=d+1$):
\begin{displaymath}
\begin{array}{ll}
\text{if } i_\nu>i>i_{\nu-1}:&\;\; Q_{\sum_{k=i+1}^d\binom{k+2}{2}-\delta+\nu},\ldots,Q_{\sum_{k=i}^d\binom{k+2}{2}-2-\delta+\nu},\\
\text{if } i=i_\nu:&\;\; Q_{\sum_{k=i+1}^d\binom{k+2}{2}-\delta+\nu+1},\ldots,Q_{\sum_{k=i}^d\binom{k+2}{2}-2-\delta+\nu}.\\
\end{array}
\end{displaymath}
\item The plane curves $C_{i_j}$ have a node germ for each $j=1,\ldots,\delta$.
\item \label{leftstringalign}If the node germ of $C_{i_j}$ is a left string, then its horizontal end aligns with a horizontal bounded edge of $C_{i_j+1}$.
\item \label{rightstringalign} If the node germ of $C_{i_j}$ is a right string, then its diagonal end aligns either with a diagonal bounded edge of $C_{i_j-1}$ or with a vertex of $C_{i_j-1}$ which is not adjacent to a diagonal edge.
\item \label{indexd} If $i_\delta=d$ then the node germ of $C_d$ is either a right string or a diagonal end of weight $2$.
\item \label{index1a} If $i_1=1$ then the node germ of $C_1$ is a left string.
\end{enumerate}
\end{definition}
Observe that each $C_i$ with $i\neq i_\nu$ passes through $\binom{i+2}{2}-1$ points, whereas each $C_{i_\nu}$ passes through $\binom{i_\nu+2}{2}-2$ points.

\begin{example}\label{ex-d2delta1} Let $d=2$ and $\delta=1$. A floor plan contains a
  tuple $(C_2,C_1)$ of tropical plane curves, where $C_2$ is a conic
  and $C_1$ a line. Furthermore, one index $i$ belongs to a curve with
  a node germ. If $i=2$, then $C_2$ passes through $Q_1,\ldots,Q_4$. By
  condition (\ref{indexd}) of definition \ref{def-floorplan}, the node
  germ in $C_2$ can either be a right string or a diagonal end of
  weight $2$. The right string would, by condition
  (\ref{rightstringalign}), have to meet a bounded edge of $C_1$ or a
  vertex which is not adjacent to a diagonal edge, which is both impossible. So the node germ at $C_2$ is a weight $2$ diagonal end. The curve $C_1$ passes through $Q_6$ and $Q_7$.
If $i=1$, $C_2$ is a smooth conic passing through $Q_1,
\ldots,Q_5$. The line $C_1$ passes through $Q_7$. By condition (\ref{index1a}), its node germ is a left string, which by condition (\ref{leftstringalign}) aligns with the unique horizontal bounded edge of $C_2$.

Altogether, we have two possible $1$-nodal floor plans of degree $2$, as shown in Figure \ref{fig-exfloorplans}.
The points $Q_5$ and $Q_8$ fix the position of the floors of degree $2$ and $1$, respectively of the surface of degree $2$ in the top floor plan.
In the bottom floor plan it is the points $Q_6$ and $Q_8$ that fix these respective floors.

\begin{figure}
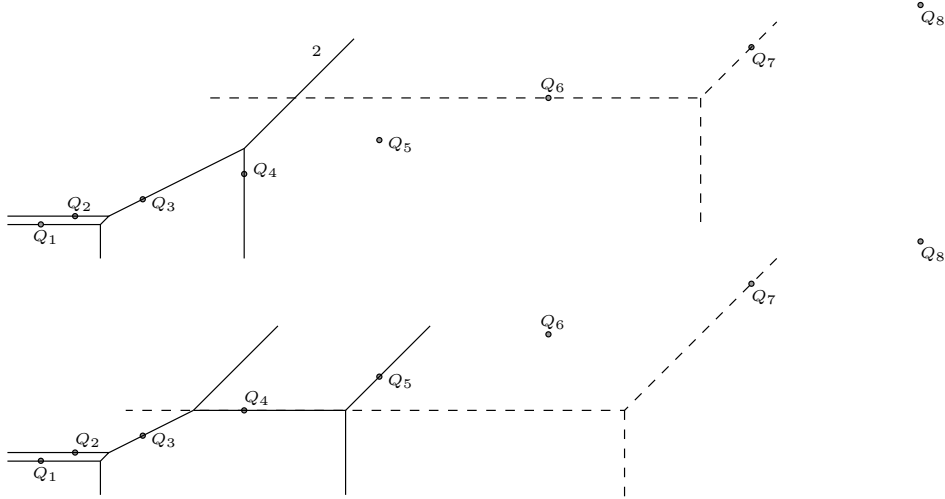
\caption{The $1$-nodal floor plans of degree $2$.}\label{fig-exfloorplans}
\end{figure}

\end{example}

\begin{definition}\label{def-complexmultP3}
Let $F=(C_d,\ldots,C_1)$ be a $\delta$-nodal floor plan of degree $d$.
For each node germ $C^*_{i_j}$ in $C_{i_j}$, we define the following local complex multiplicity $\mult_{\C}(C^*_{i_j})$:
\begin{enumerate}
\item If $C^*_{i_j}$ is dual to a parallelogram, then $\mult_{\C}(C^*_{i_j})=2$.
\item  If $C^*_{i_j}$ is the midpoint of an edge of weight $2$, then $\mult_{\C}(C^*_{i_j})=8$.
\item If $C^*_{i_j}$ is a horizontal end of weight $2$, then $\mult_{\C}(C^*_{i_j})=2\cdot(i_j+1).$
\item If $C^*_{i_j}$ is a diagonal end of weight $2$, then $\mult_{\C}(C^*_{i_j})=2\cdot(i_j-1).$
\item If $C^*_{i_j}$ is a left string, then $\mult_{\C}(C^*_{i_j})=2. $
\item If $C^*_{i_j}$ is a right string whose diagonal end aligns with a diagonal bounded edge, then $\mult_{\C}(C^*_{i_j})= 2.$
\item If $C^*_{i_j}$ is a right string whose diagonal end aligns with a vertex not adjacent to a diagonal edge, then $\mult_{\C}(C^*_{i_j})=1.$
\end{enumerate}

The \emph{complex multiplicity} of a $\delta$-nodal floor plan is defined as the product of the local multiplicities of all of its node germs,
$$\mult_{\C}(F)=\prod_{j=1}^{\delta}\mult_{\C}(C^*_{i_j}),$$
and we set $$N_{\delta,\C}^{\PP^3,\floor}(d)= \sum_F \mult_{\C}(F),$$ where the sum goes over all $\delta$-nodal floor plans of degree $d$.
\end{definition}

\begin{remark}
 Note in Definition \ref{def-complexmultP3}(1) and (2) that if $C^*_{i_j}$ is dual to a parallelogram, or the midpoint of an edge of weight $2$, then $\mult_{\C}(C^*_{i_j})$ equals twice the complex multiplicity of $C_i$, i.e.\ $2$ resp.\ $8$.
\end{remark}

\begin{example}\label{ex-floord2delta1}
Consider again the two one-nodal floor plans of degree $2$ we considered in Example \ref{ex-d2delta1}.
Both floor plans are of multiplicity $2$ by Definition \ref{def-complexmultP3}. Altogether, we have  $ N_{1,\C}^{\PP^3,\floor}(2)= 4$.

\end{example}

\begin{remark}
In Example \ref{ex-floord2delta1} (using Example \ref{ex-d2delta1}), we have determined the number $ N_{1,\C}^{\PP^3,\floor}(2)= 4$. It turns out that in this case $ N_{1,\C}^{\PP^3,\floor}(2)= 4= N_{1,\C}^{\PP^3,\trop}(2)= N_{1,\C}^{\PP^3}(2)$.  In fact, for $\delta=1$, the inequality (\ref{eq-floortrop}) is always an equality. For $d=3$ and $\delta=2$, the inequality (\ref{eq-floortrop}) is strict, as shown in \cite{BG19}. We expect the inequality to be strict as soon as $\delta>1$, since then one expects the existence of tropical surfaces for which the nodes are not apart, which we neglect with our enumeration of floor plans.
\end{remark}

\begin{construction}\label{const-surfaces}
From a $\delta$-nodal floor plan $(C_d,\ldots,C_1)$ together with the points $q_{i}$, we produce a construction pattern as in Section 3 \cite{BBLdM17}: we set $X_i=\R^2$ for all $i$, and the effective divisor $D_i=C_i$. The rational functions $f_i$ are, since they must satisfy $\divi(f_i)=C_i-C_{i-1}$, determined up to a global shift. We fix the shift such that the constructed floor meets the point between the points on $C_{i-1}$ and the points on $C_i$.
\end{construction}

\begin{proposition}\label{prop-floorplandegdlifts}
Construction \ref{const-surfaces} viewed as a map from the set of
$\delta$-nodal floor plans of degree $d$ to the set of tropical
surfaces is injective. It sends a floor plan $F$ to a $\delta$-nodal
tropical surface $S$ of degree $d$ passing through the $q_i$, and we
have $\mult_{\C}(F)=\mult_\C(S)$.

In particular, the inequality (\ref{eq-floortrop}) holds.
\end{proposition}
\begin{proof}
By construction, we obtain a tropical surface $S$ passing through the $q_i$ from a floor plan $F$. It is also obvious that the map is injective. We have to show that $S$ is $\delta$-nodal.
Each node germ $C^*_{i_j}$ is in charge of a tropical node:
\begin{enumerate}
\item If $C^*_{i_j}$ is dual to a parallelogram, the subdivision of $S$ contains a bipyramide over the parallelogram. The midpoint of the edge dual to the parallelogram is the tropicalization of a node, and there are $$\mult_{\C}(C^*_{i_j})=2$$ many ways to pick initials for the local equation of a complex surface tropicalizing to $S$ (see Lemma A.7 (1) and (2) in \cite{MMS15}).

\item If $C^*_{i_j}$ is the midpoint of an edge of weight $2$, its dual is an edge of weight $2$ adjacent to two triangles. In the subdivision dual to $S$, we have altogether $4$ bypyramides over the two triangles. According to Lemma A.1 \cite{MMS15}, there are $\mult_{\C}(C^*_{i_j})=8$ ways to pick complex lifts.
\item If $C^*_{i_j}$ is a horizontal end of weight $2$, the dual Newton subdivision of $S$ contains many pyramides involving the dual edge: one for each intersection of a diagonal end of $C_{i_j+1}$ with the edge of weight $2$. As in Lemma A.2 \cite{MMS15}, we can pick each of those intersections to fix a tropical node, and we have $2$ lifts for each choice. Altogether, the local complex multiplicity equals the number of complex lifts.
\item If $C^*_{i_j}$ is a diagonal end of weight $2$, the situation is analogous: we get to pick an intersection point with a horizontal end of $C_{i_j-1}$, and for each we have $2$ complex lifts, see Lemma A.3 \cite{MMS15}.
\item If $C^*_{i_j}$ is a left string, it has to align with a horizontal bounded edge.
The dual Newton subdivision of $S$ contains a parallelogram formed by the dual of the horizontal string edge, and the edge with which it aligns. The two vertices dual to the triangles that appear dual to $C_{i_j-1}$ form a bipyramide with the square. We have $2$ complex lifts for the coefficients according to Lemma A.7 \cite{MMS15}.
\item The same holds if $C^*_{i_j}$ is a right string whose diagonal end aligns with a diagonal bounded edge.
\item If $C^*_{i_j}$ is a right string whose diagonal end aligns with a vertex not adjacent to a diagonal edge, then the dual Newton subdivision of $S$ contains a pentatope formed by the dual of the diagonal end and the dual of the vertex with which it aligns. There is one lift according to Lemma A.10 \cite{MMS15}.
\end{enumerate}
\end{proof}

We now want to investigate the real situation. For this we
fix the sign vector with all signs positive, $s=((+)^3)^n$.
\begin{definition}\label{def-realmult1P3}
Let $F=(C_d,\ldots,C_1)$ be a $\delta$-nodal floor plan of degree $d$.
For each node germ $C^*_{i_j}$ in $C_{i_j}$, we define the following local real multiplicity $\mult_{\R,s}(C^*_{i_j})$:
\begin{enumerate}
\item If $C^*_{i_j}$ is dual to a parallelogram, it depends on the position of the parallelogram in the Newton subdivision:
\begin{itemize}
\item assume it has the vertices $(k,0)$, $(k,1)$, $(k-1,l)$ and $(k-1,l+1)$, then \begin{displaymath} \mult_{\R,s}(C^*_{i_j})=\begin{cases}2\\ 0\end{cases}\mbox{ if }\;\;\;\frac{1}{2}\cdot (3i_j+2+2k+2l)(i_j-1) \equiv \begin{cases} 1\\0\end{cases}\mbox{ modulo }2,\end{displaymath}
\item assume it has the vertices $(k,d-i_j-k)$, $(k,d-i_j-k-1)$, $(k+1,l)$ and $(k+1,l+1)$, then
\begin{displaymath} \mult_{\R,s}(C^*_{i_j})=\begin{cases}2\\ 0\end{cases}\mbox{ if }\;\;\;\frac{1}{2}\cdot (i_j+2+2l)(i_j-1) \equiv \begin{cases} 1\\0\end{cases}\mbox{ modulo }2.\end{displaymath}
 \end{itemize}
\item If $C^*_{i_j}$ is the midpoint of an edge of weight $2$, then $ \mult_{\R,s}(C^*_{i_j})=0.$
\item If $C^*_{i_j}$ is a horizontal end of weight $2$, then $\mult_{\R,s}(C^*_{i_j})=0.$
\item If $C^*_{i_j}$ is a diagonal end of weight $2$, then $\mult_{\R,s}(C^*_{i_j})=2 \cdot (i_j-1).$
\item If $C^*_{i_j}$ is a left string, then it depends on the position of the dual of the horizontal bounded edge of $C_{i_j+1}$ with which it aligns. Assume it has the vertices $(k,l)$ and $(k,l+1)$.
Then \begin{displaymath} \mult_{\R,s}(C^*_{i_j})=\begin{cases}2\\ 0\end{cases}\mbox{ if }\;\;\; i_j-k \equiv \begin{cases} 0\\1\end{cases}\mbox{ modulo }2.\end{displaymath}
\item If $C^*_{i_j}$ is a right string whose diagonal end aligns with a diagonal bounded edge, then $\mult_{\R,s}(C^*_{i_j})=0.$
\item If $C^*_{i_j}$ is a right string whose diagonal end aligns with a vertex not adjacent to a diagonal edge, then $ \mult_{\R,s}(C^*_{i_j})=1.$
\end{enumerate}

The \emph{real multiplicity} of a $\delta$-nodal floor plan is defined as the product of the local multiplicities of all of its node germs,
$$\mult_{\R,s}(F)=\prod_{j=1}^{\delta}\mult_{\R,s}(C^*_{i_j}),$$
and we let $ N_{\delta,\R,s}^{\PP^3,\floor}(d)$ be the sum of all $\delta$-nodal floor plans of degree $d$, counted with real multiplicity.
\end{definition}

\begin{example}
Revisit Example \ref{ex-d2delta1}, but now consider the real multiplicities of the two one-nodal floor plans of degree $2$. According to Definition \ref{def-realmult1P3}(4), the first floor plan has real multiplicity $2$. For the second floor plan, we use Definition \ref{def-realmult1P3}(5). Here $i_j=1$, $(k,l)=(1,0)$, so we obtain real multiplicity $2$ again. Altogether, we have  $ N_{1,\R,s}^{\PP^3,\floor}(2)=4$.
\end{example}

\begin{proposition}\label{prop-floorplandegdreallifts}
Given a floor plan $F$, the $\delta$-nodal tropical surface $S$ of degree $d$ passing through the $q_i$ which is  obtained  by Construction \ref{const-surfaces} has at least $\mult_{\R,s}(F) $ real lifts passing through the lifts $\bw$ of the $q_i$ with sign vector $s=((+)^3)^n$.
In particular, inequality (\ref{eq-floortropreal}) holds.
\end{proposition}

\begin{proof}
We compute numbers of lifts for the coefficients of a lift of $S$ near each node germ $C^*_{i_j}$:
\begin{enumerate}
\item If $C^*_{i_j}$ is dual to a parallelogram, the real local multiplicity equals the number of ways to pick real lifts, see Lemma 5.7 \cite{MMS15}. (Notice that case (2) of Lemma 5.7 \cite{MMS15} actually gives a circuit consisting of three collinear points and not a parallelogram \cite{Sin19}.)
\item[(2), (3)] If $C^*_{i_j}$ is the midpoint of an edge of weight $2$ or a horizontal end of weight $2$, the number of real lifts is hard to control (see Section 5.5 \cite{MMS15}), which is why we set the real local multiplicity to zero.
\item[(4)] If $C^*_{i_j}$ is a diagonal edge of weight $2$, we get to pick an intersection point with a horizontal end of $C_{i_j-1}$, and each of the $2$ complex lifts, is real, see Lemma 5.9 \cite{MMS15}.
\item[(5)] If $C^*_{i_j}$ is a left string, then the number of real lifts depends on the position of the dual of the horizontal bounded edge of $C_{i_j+1}$ with which it aligns. Assume it has the vertices $(k,l)$ and $(k,l+1)$.
Then there are \begin{displaymath} \mult_{\R,s}(C^*_{i_j})=\begin{cases}2\\ 0\end{cases}\mbox{ if }\;\;\; i_j-k \equiv \begin{cases} 0\\1\end{cases}\mbox{ modulo }2\end{displaymath} real lifts, see Lemma 5.6 (3) \cite{MMS15}.
\item[(6)] If $C^*_{i_j}$ is a right string whose diagonal end aligns with a diagonal bounded edge, then the contribution is of lower order in $d$ (see Remark A.9 \cite{MMS15}), so we can safely disregard them and set $ \mult_{\R,s}(C^*_{i_j})=0$ when we care for bounds for the asymptotic count.
\item[(7)] If $C^*_{i_j}$ is a right string whose diagonal end aligns with a vertex not adjacent to a diagonal edge, then there are $\mult_{\R,s}(C^*_{i_j})=1$ real lifts, see Lemma 5.4 \cite{MMS15}.
\end{enumerate}
\end{proof}

\subsection{Surfaces in $\PP^1\times \PP^2$}
Throughout this section we suppose that the integers $d,e,\delta$ satisfy $0<\delta\le\frac{1}{2}\min\{d,e\}$.
We also let $n=(d+1)\cdot \binom{e+2}{2}-1-\delta$.

Note that, under the above assumptions, the family of $\delta$-nodal surfaces of bidegree $(d,e)$ in
$\PP^1\times\PP^2$ has codimension $\delta$ in the space of all surfaces of bidegree $(d,e)$, and similarly to formula (\ref{emn1}) we have
\begin{equation}N_{\delta,\C}^{\PP^1\times\PP^2}(d,e)=
\frac{(12(d-1)(e-1)^2)^\delta}{\delta!}+O(d^{\delta-2}e^{2\delta})
+O(d^{\delta-1}e^{2\delta-1})+O(d^\delta e^{2\delta-2})\ .\label{emn1a}\end{equation}

\begin{definition}\label{def-floorplanP1P2} A \emph{$\delta$-nodal floor plan $F$ of bidegree $(d,e)$} is a tuple $(C_d,\ldots,C_0)$ of plane tropical curves $C_i$ of degree $e$, together with a choice of indices $d\geq i_\delta\geq \ldots\geq i_1\geq 0$, such that $i_{j+1}>i_j+1$ for all $j$, satisfying the following properties:
\begin{enumerate}
\item The curve $C_i$ passes through the following points  (where we
  set $i_0=-1$ and $i_{\delta+1}=d+1$):
\begin{displaymath}
\begin{array}{ll}
\text{if } i_{\nu}>i>i_{\nu-1}:& Q_{(d-i)\binom{e+2}{2}-\delta+\nu},\ldots,Q_{(d+1-i)\binom{e+2}{2}-\delta+\nu-2},\\
\text{if } i=i_\nu:& Q_{(d-i)\binom{e+2}{2}-\delta+\nu+1},\ldots,Q_{(d+1-i)\binom{e+2}{2}-\delta+\nu-2}.
\end{array}
\end{displaymath}
\item The plane curves $C_{i_j}$ have a node germ for each $j=1,\ldots,\delta$.
\item \label{leftstringaligna}If the node germ of $C_{i_j}$ is a left string, then its horizontal end aligns with a horizontal bounded edge of $C_{i_j+1}$.
\item \label{rightstringaligna} If the node germ of $C_{i_j}$ is a right string, then its diagonal end aligns either with a diagonal bounded edge of $C_{i_j-1}$ or with a vertex of $C_{i_j-1}$ which is not adjacent to a diagonal edge.
\item \label{indexdb} If $i_1=d$ then the node germ of $C_d$ is either a right string or a diagonal end of weight $2$.
\item \label{index1b} If $i_{\delta}=1$ then the node germ of $C_1$ is either a left string or a horizontal end of weight $2$.
\end{enumerate}
\end{definition}
Observe that each $C_i$ with $i\neq i_\nu$ passes through $\binom{e+2}{2}-1$ points, whereas each $C_{i_\nu}$ passes through $\binom{e+2}{2}-2$ points.

\begin{definition}\label{def-complexmultP2P1}
Let $F=(C_d,\ldots,C_0)$ be a $\delta$-nodal floor plan of bidegree $(d,e)$.
For each node germ $C^*_{i_j}$ in $C_{i_j}$, we define the following local complex multiplicity $\mult_{\C}(C^*_{i_j})$:
\begin{enumerate}
\item If $C^*_{i_j}$ is dual to a parallelogram, then $\mult_{\C}(C^*_{i_j})=2$.
\item If $C^*_{i_j}$ is the midpoint of an edge of weight $2$, then $\mult_{\C}(C^*_{i_j})=8$.
\item If $C^*_{i_j}$ is a horizontal end of weight $2$, then $\mult_{\C}(C^*_{i_j})=2e.$
\item  If $C^*_{i_j}$ is a diagonal end of weight $2$, then $\mult_{\C}(C^*_{i_j})=2e.$
\item If $C^*_{i_j}$ is a left string, then $\mult_{\C}(C^*_{i_j})=2. $
\item If $C^*_{i_j}$ is a right string whose diagonal end aligns with a diagonal bounded edge, then $\mult_{\C}(C^*_{i_j})= 2.$
\item If $C^*_{i_j}$ is a right string whose diagonal end aligns with a vertex not adjacent to a diagonal edge, then $\mult_{\C}(C^*_{i_j})=1.$
\end{enumerate}

The \emph{complex multiplicity} of a $\delta$-nodal floor plan is defined as the product of the local multiplicities of all of its node germs,
$$\mult_{\C}(F)=\prod_{j=1}^{\delta}\mult_{\C}(C^*_{i_j}),$$
and we set $$N_{\delta,\C}^{\PP^1\times\PP^2,\floor}(d,e)= \sum_F \mult_{\C}(F),$$ where the sum goes over all $\delta$-nodal floor plans of bidegree $(d,e)$.
\end{definition}

\begin{proposition}\label{prop-floorplandegdelifts}
Construction \ref{const-surfaces} viewed as a map from the set of
$\delta$-nodal floor plans of bidegree $(d,e)$ to the set of tropical
surfaces is injective. It sends a floor plan $F$ to a $\delta$-nodal
tropical surface $S$ of bidegree $(d,e)$ passing through the $q_i$,
and we have $\mult_{\C}(F)=\mult_\C(S)$.
In particular,

\begin{equation} N_{\delta,\C}^{\PP^1\times\PP^2,\floor}(d,e)\leq N_{\delta,\C}^{\PP^1\times\PP^2,\trop}(d,e).\end{equation}

\end{proposition}
\begin{proof}
The proof is analogous to Proposition \ref{prop-floorplandegdlifts}. The only difference is that a horizontal or diagonal end of weight $2$ intersects $e$ edges of the previous or next floor, respectively.
\end{proof}

Again we want to investigate the real situation and fix the all positive sign vector,   $s=((+)^3)^n$.

For the following, we assume $e\equiv0\mod4$. This restriction is important for case (1) of Definition \ref{def-realmult1P1P2} below.

\begin{definition}\label{def-realmult1P1P2}
Let $F=(C_d,\ldots,C_0)$ be a $\delta$-nodal floor plan of bidegree $(d,e)$ where $e\equiv0\mod4$.
For each node germ $C^*_{i_j}$ in $C_{i_j}$, we define the following local real multiplicity $\mult_{\R,s}(C^*_{i_j})$:
\begin{enumerate}
\item If $C^*_{i_j}$ is dual to a parallelogram, 
then $ \mult_{\R,s}(C^*_{i_j})=2$.
\item If $C^*_{i_j}$ is the midpoint of an edge of weight $2$, then $ \mult_{\R,s}(C^*_{i_j})=0$.
\item If $C^*_{i_j}$ is a horizontal end of weight $2$, then $\mult_{\R,s}(C^*_{i_j})=0$.
\item If $C^*_{i_j}$ is a diagonal end of weight $2$, then $\mult_{\R,s}(C^*_{i_j})=2e$.
\item If $C^*_{i_j}$ is a left string, then it depends on the position of the dual of the horizontal bounded edge of $C_{i_j+1}$ with which it aligns. Assume it has the vertices $(k,l)$ and $(k,l+1)$.
Then \begin{displaymath} \mult_{\R,s}(C^*_{i_j})=\begin{cases}2\\ 0\end{cases}\mbox{ if }\;\;\; i_j-k \equiv \begin{cases} 0\\1\end{cases}\mbox{ modulo }2.\end{displaymath}
\item If $C^*_{i_j}$ is a right string whose diagonal end aligns with a diagonal bounded edge, then $\mult_{\R,s}(C^*_{i_j})=0.$
\item If $C^*_{i_j}$ is a right string whose diagonal end aligns with a vertex not adjacent to a diagonal edge, then $ \mult_{\R,s}(C^*_{i_j})=1.$
\end{enumerate}

The \emph{real multiplicity} of a $\delta$-nodal floor plan is defined as the product of the local multiplicities of all of its node germs,
$$\mult_{\R,s}(F)=\prod_{j=1}^{\delta}\mult_{\R,s}(C^*_{i_j}),$$
and we set $$N_{\delta,\R, s}^{\PP^1\times\PP^2,\floor}(d,e)= \sum_F \mult_{\R, s}(F),$$ where the sum goes over all $\delta$-nodal floor plans of bidegree $(d,e)$.

\end{definition}

\begin{proposition}\label{prop-floorplandegdrealliftsP1P2}
Given a floor plan $F$, the $\delta$-nodal tropical surface $S$ of bidegree $(d,e)$, where $e\equiv0\mod4$, passing through the $q_i$ 
has at least $\mult_{\R,s}(F) $ real lifts.
In particular, $$N_{\delta,\R,s}^{\PP^1\times\PP^2,\floor}(d,e)\leq N_{\delta,\R}^{\PP^1\times\PP^2,\trop}((d,e),\bw)\ .$$
\end{proposition}

\begin{proof}
The proof is analogous to that of Proposition
\ref{prop-floorplandegdreallifts}, and we have to confirm the
individual real multiplicities of each node. It follows that the real
multiplicities in Definition \ref{def-realmult1P1P2} (3-7) match the
numbers of real lifts in the same way as the real multiplicities in
Definition \ref{def-realmult1P3} (4-7) treated in the proof of Proposition \ref{prop-floorplandegdreallifts}. In the remaining case of
$C^*_{i_j}$ dual to a parallelogram and $e\equiv0\mod4$ (see
Definition \ref{def-realmult1P1P2}), we claim that there are two real lifts. Indeed, as indicated in the proof of \cite[Lemma 4.8]{MMS15}, the question reduces to the existence of two real solutions with respect to the variable $z_{1,0}$ in the system \cite[Formula (36)]{MMS15}, which we reproduce here in the form
\begin{equation}a_{1,i',j'}z_{20}^{i'}z_{30}^{j'}-a_{-1,m',n'}z_{20}^{m'}z_{30}^{n'}z_{10}^{-2}=0\
  .\label{ne1}\end{equation}

Assume that the parallelogram dual to $C^*_{i_j}$ has vertices
$$(i_j,k,l),\ (i_j,k,l+1), (i_j,k+1,0),\ (i_j,k+1,1)\ .$$ This circuit is associated with two additional points (cf. the second paragraph in the proof of
\cite[Lemma 5.7]{MMS15})
\begin{equation}\bw'=(i_j-1,e,0),\quad\bw''=(i_j+1,0,0)\ .\label{ne2}\end{equation} Performing an affine-integral transformation that takes the given parallelogram into the following one
\begin{equation}\conv\{(0,0,0),\ (0,1,0),\ (0,0,1), (0,1,1)\}\ ,\label{ne3}\end{equation} we obtain
$$\bw'=(-1,e-k,l(e-k-1)),\quad\bw''=(1,-k,-l(k+1))\ .$$ It follows that in (\ref{ne1}) we have
$$i'=-k,\quad j'=-l(k+1),\quad m'=e-k,\quad n'= l(e-k-1)\ .$$ The parameters $z_{20},z_{30}$ can be found from the equations
$$a_{001}+a_{011}z_{20}=0,\quad a_{010}+a_{011}z_{30}=0$$
(cf. \cite[System (36)]{MMS15}). Due to the relation
$$a_{000}a_{011}=a_{010}a_{001}$$
(cf. the first formula in \cite[System (36)]{MMS15}), we have
\begin{equation}z_{30}=-\frac{a_{000}}{a_{001}}>0\ .\label{e1807}\end{equation} Indeed, the parameters $a_{000}$ and $a_{001}$ come from the coefficients at the points $(i_j,k,l)$, $(i_j,k,l+1)$ of the original circuit, and they must have different signs in view of the totally positive sign vector $s$ and the fact that the aforementioned points are joined by one segment of the lattice path. Next, we notice that the difference of the exponents of $z_{20}$ in (\ref{ne1}) equals $m'-i'=e-k-(-k)=e\equiv0\mod4$.
Hence, under the condition $e\equiv0\mod4$, equation (\ref{ne1}) has two real solutions in $z_{10}$ as long as $a_{1,i',j'}a_{-1,m',n'}>0$. Due to our choice of the totally positive sign vector $s$, the latter relation is equivalent to the evenness of the number of the edges of the lattice path joining the points $\bw'$ and $\bw''$. Since this number equals $\frac{1}{2}e(e+3)+2\equiv0\mod2$, we obtain the desired claim.

Assume that the parallelogram dual to $C^*_{i_j}$ has vertices
$$(i_j,k,e-k-1),\ (i_j,k,e-k), (i_j,k+1,l),\ (i_j,k+1,l+1)$$ and is associated with the same points (\ref{ne2}).
By an affine integral transformation we bring the given parallelogram to the form (\ref{ne3}) and obtain
$$\bw'=(-1,e-k,(e-k)(e-k-l-2)+1),\quad\bw''=(1,-k,-k(e-k-l-1)-e+k+1)\ ,$$ which yields in (\ref{ne1})
$$i'=-k,\quad j'=-k(e-k-l-1)-e+k+1,\quad m'=e-k,\quad n'=(e-k)(e-k-l-2)+1\ .$$ As in the preceding case, we have $z_{30}>0$, which finally provides two real solutions in (\ref{ne1}) as long as the number of the edges of the lattice path joining the points $\bw'$ and $\bw''$ is even. Again this number is $\frac{1}{2}e(e+3)+2\equiv0\mod2$, we obtain the desired claim.
\end{proof}

\subsection{Surfaces in $\PP^1\times \PP^1\times \PP^1$}
Throughout this section we suppose that the integers $d,e,f,\delta$ satisfy $0<\delta<\frac{1}{2}\min\{d,e,f\}$.
We also let $n=(d+1)(e+1)(f+1)-1-\delta$.

Note that, under the above assumptions, the family of $\delta$-nodal surfaces of tridegree $(d,e,f)$ in
$\PP^1\times\PP^1\times\PP^1$ has codimension $\delta$ in the space of all surfaces of tridegree $(d,e,f)$, and similarly to formulas (\ref{emn1}), (\ref{emn1a}) we have
$$N_{\delta,\C}^{\PP^1\times\PP^1\times\PP1}(d,e,f)=
\frac{(24(d-1)(e-1)(f-1))^\delta}{\delta!}+O(d^{\delta-2}e^\delta f^\delta)+
+O(d^\delta e^{\delta-2}f^\delta)+O(d^\delta e^\delta f^{\delta-2})$$
$$+O(d^{\delta-1}e^{\delta-1}f^\delta)+O(e^{\delta-1}e^\delta f^{\delta-1})+O(d^\delta
e^{\delta-1}f^{\delta-1})\ .$$

\begin{definition}\label{def-floorplanP1P1P1} A \emph{$\delta$-nodal floor plan $F$ of tridegree $(d,e,f)$} is a tuple $(C_d,\ldots,C_0)$ of plane tropical curves $C_i$ of bidegree $(e,f)$, together with a choice of indices $d\geq i_\delta\geq \ldots\geq i_1\geq 0$, such that $i_{j+1}>i_j+1$ for all $j$, satisfying the following properties:
\begin{enumerate}
\item The curve $C_i$ passes through the following points (where we
  set $i_0=-1$ and $i_{\delta+1}=d+1$):
\begin{displaymath}
\begin{array}{ll}
\text{if } i_{\nu}>i>i_{\nu-1}:& Q_{(d-i)(e+1)(f+1)-\delta+\nu},\ldots,Q_{(d+1-i)(e+1)(f+1)-\delta+\nu-2},\\
\text{if } i=i_\nu:& Q_{(d-i)(e+1)(f+1)-\delta+\nu+1},\ldots,Q_{(d+1-i)(e+1)(f+1)-\delta+\nu-2}.
\end{array}
\end{displaymath}

\item The plane curves $C_{i_j}$ have a node germ for each $j=1,\ldots,\delta$.
\item \label{leftstringalignb}If the node germ of $C_{i_j}$ is a left string, then its horizontal end aligns with a horizontal bounded edge of $C_{i_j+1}$.
\item \label{rightstringalignb} If the node germ of $C_{i_j}$ is a right string, then its horizontal end aligns with a horizontal bounded edge of $C_{i_j-1}$.
\item \label{indexdc} If $i_1=d$ then the node germ of $C_d$ is either a right string, or a right horizontal or up vertical end of weight $2$.
\item \label{index1c} If $i_{\delta}=1$ then the node germ of $C_1$ is either a left string or a left horizontal or down vertical end of weight $2$.
\end{enumerate}
\end{definition}
Observe that each $C_i$ with $i\neq i_\nu$ passes through $ef+e+f$ points, whereas each $C_{i_\nu}$ passes through $ef+e+f-1$ points.

\begin{definition}\label{def-complexmultP1P1P1}
Let $F=(C_d,\ldots,C_0)$ be a $\delta$-nodal floor plan of tridegree $(d,e,f)$.
For each node germ $C^*_{i_j}$ in $C_{i_j}$, we define the following local complex multiplicity $\mult_{\C}(C^*_{i_j})$:
\begin{enumerate}
\item If $C^*_{i_j}$ is dual to a parallelogram, then $\mult_{\C}(C^*_{i_j})=2$.
\item If $C^*_{i_j}$ is the midpoint of an edge of weight $2$, then $\mult_{\C}(C^*_{i_j})=8$.
\item If $C^*_{i_j}$ is a horizontal end of weight $2$, then $\mult_{\C}(C^*_{i_j})=2f.$
\item If $C^*_{i_j}$ is a vertical end of weight $2$, then $\mult_{\C}(C^*_{i_j})=2e.$
\item If $C^*_{i_j}$ is a left string, then $\mult_{\C}(C^*_{i_j})=2. $
\item If $C^*_{i_j}$ is a right string, then $\mult_{\C}(C^*_{i_j})=2. $
\end{enumerate}

The \emph{complex multiplicity} of a $\delta$-nodal floor plan is defined as the product of the local multiplicities of all of its node germs,
$$\mult_{\C}(F)=\prod_{j=1}^{\delta}\mult_{\C}(C^*_{i_j}),$$
and we set $$N_{\delta,\C}^{\PP^1\times\PP^1\times\PP^1,\floor}(d)= \sum_F \mult_{\C}(F),$$ where the sum goes over all $\delta$-nodal floor plans of tridegree $(d,e,f)$.
\end{definition}
Notice that different to Definitions \ref{def-complexmultP3} and \ref{def-complexmultP2P1}, a right string cannot align with a vertex which is not adjacent to a horizontal edge, which is why we do not have case (7) in $\PP^1\times\PP^1\times\PP^1$ that we had in those two Definitions.

\begin{proposition}\label{prop-floorplandegdeflifts}
Construction \ref{const-surfaces} viewed as a map from the set of
$\delta$-nodal floor plans of tridegree $(d,e,f)$ to the set of tropical surfaces is injective. It sends a floor plan $F$ to a $\delta$-nodal tropical surface $S$ of tridegree $(d,e,f)$ passing through the $q_i$, and we have $\mult_{\C}(F)=\mult_\C(S)$.
In particular, \begin{equation} N_{\delta,\C}^{\PP^1\times\PP^1\times\PP^1,\floor}(d,e,f)\leq N_{\delta,\C}^{\PP^1\times\PP^1\times\PP^1,\trop}(d,e,f).\end{equation}

\end{proposition}
\begin{proof}
The proof is analogous to Proposition \ref{prop-floorplandegdlifts}. The only difference is that a horizontal end of weight $2$ intersects $f$ edges of the previous or next floor, and a vertical end of weight $2$ intersects $e$ edges of the next floor, and that a right string cannot align with a vertex which is not adjacent to a horizontal edge.
\end{proof}

We now turn again to the real situation with a somewhat more subtle
choice of signs.

\begin{definition}\label{signvectorP1P1P1}
We define the sign vector $s=\{(+,+,\eps_r)\}_{r=1,...,n}$ so that $\eps_r=+$ everywhere, except for
$$r=f+2+2\lambda f+\mu(e+1)(f+1),\quad \lambda=0,...,\left[\frac{e-1}{2}\right]-\theta,\ \mu=1,...,d-2\ ,$$ where
$\theta=0$ or $1$ is chosen to satisfy the condition
$$ef+e+f+\left[\frac{e-1}{2}\right]-\theta\equiv0\mod2\ .$$
\end{definition}

\begin{remark}\label{nr1}
For the reader's convenience, we comment on the meaning of the formulas for the sign vector $s$ in Definition \ref{signvectorP1P1P1}.
Extend the given configuration $\bw$ with $\delta$ extra points following the rule in the first paragraph of Section \ref{nsec3}.
Take the corresponding lattice path which goes through all the integral points of the polytope (\ref{nepolytope}). In each slice given by fixing the first coordinate $i=\const\ge1$ and for each odd $k=2\lambda+1$ (with $\lambda$ as in Definition \ref{signvectorP1P1P1}), we consider the fragment of the lattice path joining the points
$(i,k,0)$ and $(i,k,f)$ and pick the bottom segment. This segment is associated with a point in the given configuration, and for this point, we set the signs $(+,+,-)$. All the remaining elements of the sign vector $s$ are $(+,+,+)$.
When we consider the actual lattice path omitting $\delta$ integral points in the polytope, we, first, exile the extra $\delta$ points in the configuration and, second, we notice that each ``negative" segment of the lattice path is shifted upward by the number of preceding omitted points. The idea behind this choice of the sign vector is that it allows one
to obtain more real $\delta$-nodal surfaces than for the totally positive sign vector (see details in the proof of Proposition
\ref{prop-floorplandegdfrealliftsP1P1P1}).
\end{remark}

\begin{definition}\label{def-realmult1P1P1P1}
Let $F=(C_d,\ldots,C_0)$ be a $\delta$-nodal floor plan of tridegree $(d,e,f)$.
For each node germ $C^*_{i_j}$ in $C_{i_j}$, we define the following local real multiplicity $\mult_{\R,s}(C^*_{i_j})$:
\begin{enumerate}
\item If $C^*_{i_j}$ is dual to a parallelogram either with vertices
\begin{equation}(i_j,k,l),\ (i_j,k,l+1),\ (i_j,k+1,0),\ (i_j,k+1,1),
\label{ne4}\end{equation}
or
\begin{equation}(i_j,k,f-1),\ (i_j,k,f),\ (i_j,k+1,l),\ (i_j,k+1,l+1),
\label{ne5}\end{equation} and $e$ is even, then $\mult_{\R,s}(C^*_{i_j})=2$.
\item If $C^*_{i_j}$ is the midpoint of an edge of weight $2$, then $ \mult_{\R,s}(C^*_{i_j})=0.$
\item If $C^*_{i_j}$ is a horizontal end of weight $2$, then $\mult_{\R,s}(C^*_{i_j})=0.$
\item If $C^*_{i_j}$ is a vertical end of weight $2$, dual to one of the segments
$[(i_j,k-1,0),(i_j,k+1,0)]$ or $[(i_j,k-1,f),(i_j,k+1,f)]$
where $k\le2\left[\frac{e-1}{2}\right]-2\theta+1,$
then $\mult_{\R,s}(C^*_{i_j})=2f$.
\item If $C^*_{i_j}$ is a left string, then it depends on the position of the dual of the horizontal bounded edge of $C_{i_j+1}$ with which it aligns. Assume it has the vertices $(i_j+1,k,l)$ and $(i_j+1,k,l+1)$.
Then \begin{displaymath} \mult_{\R,s}(C^*_{i_j})=\begin{cases}2\\ 0\end{cases}\mbox{ if }\;\;\; f+k \equiv \begin{cases} 0\\1\end{cases}\mbox{ modulo }2.\end{displaymath}
\item If $C^*_{i_j}$ is a right string, then it depends on the position of the dual of the horizontal bounded edge of $C_{i_j+1}$ with which it aligns. Assume it has the vertices $(i_j-1,k,l)$ and $(i_j-1,k,l+1)$.
Then \begin{displaymath} \mult_{\R,s}(C^*_{i_j})=\begin{cases}2\\ 0\end{cases}\mbox{ if }\;\;\; f+k \equiv \begin{cases} 0\\1\end{cases}\mbox{ modulo }2.\end{displaymath}

\end{enumerate}

The \emph{real multiplicity} of a $\delta$-nodal floor plan is defined as the product of the local multiplicities of all of its node germs,
$$\mult_{\R,s}(F)=\prod_{j=1}^{\delta}\mult_{\R,s}(C^*_{i_j}),$$
and we let $$N_{\delta,\R, s}^{\PP^1\times\PP^1\times\PP^1,\floor}(d,e,f)= \sum_F \mult_{\R, s}(F),$$ where the sum goes over all $\delta$-nodal floor plans of tridegree $(d,e,f)$.
\end{definition}

\begin{proposition}\label{prop-floorplandegdfrealliftsP1P1P1} Suppose that $e$ is even.
Given a floor plan $F$, the $\delta$-nodal tropical surface $S$ of tridegree $(d,e,f)$ passing through the $q_i$ we obtain by Construction \ref{const-surfaces} at least $\mult_{\R,s}(F) $ real lifts.
In particular, $$N_{\delta,\R,s}^{\PP^1\times\PP^1\times\PP^1,\floor}(d,e,f)\leq N_{\delta,\R}^{\PP^1\times\PP^1\times\PP^1}((d,e,f),\bw).$$
\end{proposition}

\begin{proof} As in Propositions \ref{prop-floorplandegdreallifts} and \ref{prop-floorplandegdrealliftsP1P2}
the real multiplicity of a floor plan equals the product of the numbers of real lifts of each node $C_{i_j}^*$.

We start with the following observation: in any lattice path corresponding to $\delta$-nodal floor plans under consideration, the segments that are associated with the sign elements $(+,+,-)\in s$ are of the following form (cf. Remark \ref{nr1})
$$[(i,k,l_{ik}),(i,k,l_{ik}+1)],\quad\text{where}\ \begin{cases}&0\le i\le d,\\
&1\le k\le2[(e-1)/2]-2\theta+1,\ k\equiv1\mod2,\\
&0\le l_{ik}\le\delta\end{cases}$$

Consider the node as in Definition \ref{def-realmult1P1P1P1} (1). Let it be dual to the parallelogram (\ref{ne4}). This parallelogram is associated with the points
$\bw'=(i_j-1,e,f)$ and $\bw''=(i_j+1,0,0)$. An affine-integral transformation that takes the parallelogram to the canonical form (\ref{ne3}) brings $\bw',\bw''$ to the form
$$\bw'=(-1,e-k,l(e-k)+f-l)=:(-1,m',n'),\quad\bw''=(1,-k,-l(k+1))=:(1,i',j')\ .$$ The key equation (\ref{ne1}) turns into
\begin{equation}a_{1,i',j'}z_{20}^{-e}-a_{-1,m',n'}z_{30}^{le+f}z_{10}^{-2}=0\ .\label{ne6}\end{equation} In view of relation (\ref{e1807}) which holds here in the same way as in the proof of Proposition \ref{prop-floorplandegdrealliftsP1P2}, equation (\ref{ne6}) has two real solutions with respect to $z_{10}$, since $e$ is even and $a_{1,i',j'}a_{-1,m',n'}>0$. The latter relations comes from the fact that the number of sign changes along the fragment of the lattice path between the points $\bw',\bw''$ equals
$$ef+e+f-\left(\left[\frac{e-1}{2}\right]-\theta\right)\equiv0\mod2\ .$$
In the case of the parallelogram (\ref{ne5}), we bring it to the canonical form (\ref{ne3}) and obtain
$$\bw'=(-1,e-k,(e-k)(f-l-1)+l)=:(-1,m',n'),\quad\bw''=(1,-k,k(l+1-f)-f+1)=:(1,i',j')\ .$$ Hence, the key equation (\ref{ne1}) takes the form
$$a_{1,i',j'}z_{20}^{-e}-a_{-1,m',n'}z_{30}^{e(f-l-1)+f}z_{10}^{-2}=0\ ,$$
which for the above reason yields two real solutions.

In the case of Definition \ref{def-realmult1P1P1P1} (5), the node is associated with the parallelogram circuit
$$\conv\{(i_j-1,k,l),\ ((i_j-1,k,l+1),\ (i_j,0,0),\ (i_j,0,1)\}$$ and two points
$$\bw'=(i_j-1,k-1,f),\quad\bw''=(i_j-1,k+1,0)\ .$$ Bringing the parallelogram to the canonical form
(\ref{ne3}) by an affine-integral transformation, we obtain
$$\bw'=(-1,0,f-2l)=:(-1,m',n'),\quad\bw''=(1,0,0)=:(1,i',j'),$$ and the key equation (\ref{ne1}) in the form
$$a_{1,i',j'}-a_{-1,m',n'}z_{30}^{f-2l}z_{10}^{-2}=0\ .$$ To get two real solutions we need the number of the sign changes along the fragment of the lattice path between the points $\bw',\bw''$ to be even, and this number (cf. Definition
\ref{signvectorP1P1P1}) has the parity of $f+k$. In the same manner we
treat the case of Definition \ref{def-realmult1P1P1P1} (6).

At last, the case of Definition \ref{def-realmult1P1P1P1} (4) is governed by \cite[Lemma 4.9(1) and Lemma 5.9]{MMS15}. An upward oriented edge of weight $2$ intersects $f$ horizontal left-oriented ends of the floor $C_{i_j+1}$, and each intersection point yields $2$ real lifts. Indeed, the sufficient condition
(cf. the first paragraph of the proof of \cite[lemma 5.9]{MMS15}) for
that is that the number of sign changes in the fragment of the lattice path joining the endpoints of the circuit is even. This fragment contains
$2f+1$ arcs, each one yielding a sign change except for exactly one arc (see Definition \ref{signvectorP1P1P1}). Similarly we obtain $2f$ real lifts for each downward oriented edge of weight $2$.
\end{proof}

\section{Asymptotics and lower bounds for counts of floor plans for surfaces}\label{sec-asympt}

\begin{theorem}\label{thm-floorplanasymptoticP3}
The number $ N_{\delta,\C}^{\PP^3,\floor}(d)$ of $\delta$-nodal floor plans for surfaces of degree $d$ satisfies $$N_{\delta,\C}^{\PP^3,\floor}(d)= (4d^3)^\delta/\delta!+O(d^{3\delta-1}).$$
\end{theorem}
\begin{proof}
Given a floor plan, we associate an ordered partition $(\delta_1,\delta_2,\delta_3)$ of $\delta$, where, for $i=1,2,3$, $\delta_i$ is the number of node germs of type (i) in Definition \ref{def-nodegerm}.
We first count floor plans for a fixed partition $(\delta_1,\delta_2,\delta_3)$.
Relabeling the tuple of indices, we let $i_{1}>\ldots>i_{\delta_1}$ be
the indices of the curves carrying a node germ of type (1), $i_{\delta_1+1}>\ldots>i_{\delta_1+\delta_2}$ the indices of the curves carrying a node germ of type (2) and $i_{\delta_1+\delta_2+1}>\ldots>i_\delta$ the indices of curves with a node germ of type (3). Let us momentarily fix such a choice of indices, and consider the contribution of all floor plans for our fixed partition and for the fixed choice of indices. This contribution is polynomial in the $i_j$. Since we will later sum over all possible $i_j$, we only need to take the leading order into account for our asymptotic count.

By Proposition \ref{prop-complexcurvecount}, the number of $C_{i_j}$ satisfying the conditions with a node germ of type (1) is $3i_j^2+O(i_j)$, and the complex multiplicity of such a node germ is defined to be twice the multiplicity of the curve (see Definition \ref{def-complexmultP3}). Thus, for each $i_j$, $j=1,\ldots,\delta_1$, we obtain a factor of $6i_j^2$.

The number of $C_{i_j}$ with a node germ of type (2) is $2i_j$, since
we have that many possible locations for horizontal or diagonal ends
of weight $2$. The multiplicity for each one is $2i_j+O(1)$ by Definition \ref{def-complexmultP3}, so altogether we obtain a factor of $4i_j^2$.

Finally, consider a curve $C_{i_j}$ with a left string.
It can align with any horizontal bounded edge of $C_{i_j+1}$, and since the latter is a smooth tropical plane curve of degree $i_j+1$ satisfying horizontally stretched point conditions, there are $i_j+\ldots+1=\frac{1}{2}i_j^2+O(i_j)$ such edges. For each, the local complex multiplicity is $2$, and so we have a contribution of $i_j^2$ for these. If $C_{i_j}$ has a right string, it can either meet a diagonal bounded edge of $C_{i_j-1}$, of which there are $i_j-2$ many, or a vertex not adjacent to a diagonal edge, of which there are $2(i_j-3+i_j-4+\ldots+1)=i_j^2+O(i_j)$ many. We can see that the alignments with a diagonal bounded edge do not produce the leading degree, so we can neglect them. Altogether, we obtain a factor of $i_j^2+i_j^2=2i_j^2$ for all possibilities for $C_{i_j}$ which have a string, counted with multiplicity.

Thus, the leading term of the contribution of all floor plans with fixed partition and fixed choice of indices as above is
$$ 6^{\delta_1}\cdot \prod_{j=1}^{\delta_1} i_j^2\cdot 4^{\delta_2}\cdot \prod_{j=\delta_1+1}^{\delta_1+\delta_2} i_j^2\cdot 2^{\delta_3}\cdot \prod_{j=\delta_1+\delta_2+1}^{\delta}i_j^2.$$

We have to sum this expression over all choices of indices. By
definition of floor plans (see \ref{def-floorplan}), the indices have
to be sufficiently apart. Simplifying our summation ranges in a way
that this is not taken into account however, we do not change the
leading term. Also, the computations from above are not valid for the
indices $1,d$, since there are special requirements in the definition
of a floor plan. Again, neglecting those does not change the leading term.
 We conclude that the leading term of the contribution of floor plans with partition $(\delta_1,\delta_2,\delta_3)$ equals the leading term of
$$
\sum_{i_1=1}^{d} \sum_{i_2=1}^{i_1} \ldots \sum_{i_{\delta_1}=1}^{i_{\delta_1-1}} \sum_{i_{\delta_1+1}=1}^{d} \ldots \sum_{i_{\delta_1+\delta_2}=1}^{i_{\delta_1+\delta_2-1}} \sum_{i_{\delta_1+\delta_2+1}=1}^{d} \ldots \sum_{i_{\delta}=1}^{i_{\delta-1}}
 6^{\delta_1}\cdot \prod_{j=1}^{\delta_1} i_j^2\cdot 4^{\delta_2}\cdot \prod_{j=\delta_1+1}^{\delta_1+\delta_2} i_j^2\cdot 2^{\delta_3}\cdot \prod_{j=\delta_1+\delta_2+1}^{\delta}i_j^2.
$$

We compute this, using Faulhaber's formula $\sum_{h=1}^d h^p= \frac{1}{p+1} d^{p+1}+O(d^{p})$, to be
 $$6^{\delta_1} \left(\frac{1}{3}\right)^{\delta_1} \frac{1}{\delta_1!}\cdot  4^{\delta_2}\left(\frac{1}{3}\right)^{\delta_2} \frac{1}{\delta_2!}\cdot 2^{\delta_3}\left(\frac{1}{3}\right)^{\delta_3}\frac{1}{\delta_3!} d^{3\delta}= 2^{\delta_1+\delta_3}\cdot 4^{\delta_2}\cdot \left(\frac{1}{3}\right)^{\delta_2+\delta_3}\cdot \frac{1}{\delta_1!\delta_2!\delta_3!}\cdot d^{3\delta}.$$

What remains to be done is to sum the coefficient of this expression over all $(\delta_1,\delta_2,\delta_3)$. Taking $\delta_1+\delta_2+\delta_3=\delta$ into account, we obtain

\begin{align*}
&\sum_{\delta_1=0}^{\delta}\sum_{\delta_2=0}^{\delta-\delta_1} 2^{\delta}2^{\delta_2}\left(\frac{1}{3}\right)^{\delta-\delta_1}\cdot \frac{1}{\delta_1!\delta_2!(\delta-\delta_1-\delta_2)!}
= 2^{\delta}\sum_{\delta_1=0}^{\delta} \left(\frac{1}{3}\right)^{\delta-\delta_1}\frac{1}{\delta_1!}\sum_{\delta_2=0}^{\delta-\delta_1} 2^{\delta_2} \frac{1}{\delta_2!(\delta-\delta_1-\delta_2)!}\\
=& 2^{\delta}\sum_{\delta_1=0}^{\delta} \left(\frac{1}{3}\right)^{\delta-\delta_1}\frac{1}{\delta_1!(\delta-\delta_1)!}
\sum_{\delta_2=0}^{\delta-\delta_1} 2^{\delta_2} \binom{\delta-\delta_1}{\delta_2}
= 2^{\delta}\sum_{\delta_1=0}^{\delta} \left(\frac{1}{3}\right)^{\delta-\delta_1}\frac{1}{\delta_1!(\delta-\delta_1)!} (1+2)^{\delta-\delta_1}\\
=& 2^{\delta}\sum_{\delta_1=0}^{\delta} \frac{1}{\delta_1!(\delta-\delta_1)!}
= \frac{2^\delta}{\delta!}\sum_{\delta_1=0}^{\delta} \binom{\delta}{\delta_1}
= \frac{2^\delta}{\delta!}(1+1)^\delta = \frac{4^\delta}{\delta!}.
\end{align*}

\end{proof}

We can adapt this count, producing lower bounds for numbers of real surfaces:

\begin{theorem}\label{thm-realfloorplanasymptoticP3positive}
Let $s=((+)^3)^n$ be the sign vector with all signs positive. Then the number of $\delta$-nodal floor plans of degree $d$, counted with real multiplicity w.r.t. $s$, satisfies:
$$ N_{\delta,\R,s}^{\PP^3,\floor}(d) =\frac{1}{\delta!}\left(\frac{3}{2}d^3\right)^\delta+O(d^{3\delta-1}).$$\end{theorem}
\begin{proof}
The proof is analogous to the proof of Theorem \ref{thm-floorplanasymptoticP3}, we just have to adapt the numbers to the real multiplicities.

Consider curves $C_{i_j}$ with a node of type (1).
Going through Definition \ref{def-realmult1P3}, we can see that for node germs which are midpoints of an edge of weight $2$, the real multiplicity is $0$. From the proof of Proposition \ref{prop-complexcurvecount}, we can see that these curves account for $2i_j^2+O(i_j)$ in the total complex count. The ones with node germ dual to a parallelogram contribute the remaining $i_j^2+O(i_j)$ in the complex count, but here, only half of them come with a nonzero real multiplicity, being $2$.  Thus, for each $i_j$, $j=1,\ldots,\delta_1$, we obtain a factor of $i_j^2$.

The number of $C_{i_j}$ with a node germ of type (2) is $2i_j$, since we have that many possible locations for horizontal or diagonal ends of weight $2$. However, now only the ones with a diagonal end contribute by Definition \ref{def-realmult1P3},
 and their multiplicity for each is $2i_j+O(1)$, so altogether we obtain a factor of $2i_j^2$.

Finally, consider a curve $C_{i_j}$ with a left string.
It can align with any horizontal bounded edge of $C_{i_j+1}$, but only half of the choices lead to a nonzero multiplicity of $2$.
Thus, we have a contribution of $\frac{1}{2}i_j^2$ for these. If $C_{i_j}$ has a right string, we can neglect alignments with a diagonal bounded edge as before and only consider the possibilities to meet a vertex not adjacent to a diagonal edge of $C_{i_j-1}$, of which there are $i_j^2+O(i_j)$ many, each counting with multiplicity one. Altogether, we obtain a factor of $\frac{1}{2}i_j^2+i_j^2=\frac{3}{2}i_j^2$ for all possibilities for $C_{i_j}$ which have a string, counted with multiplicity.

Thus, the leading term of the contribution to the real count of all floor plans with fixed partition and fixed choice of indices as above is
$$ 1^{\delta_1}\cdot \prod_{j=1}^{\delta_1} i_j^2\cdot 2^{\delta_2}\cdot \prod_{j=\delta_1+1}^{\delta_1+\delta_2} i_j^2\cdot \frac{3}{2}^{\delta_3}\cdot \prod_{j=\delta_1+\delta_2+1}^{\delta}i_j^2.$$

We have to sum this expression over all choices of indices as above, and using Faulhaber's formula again we obtain
 $$1^{\delta_1} \left(\frac{1}{3}\right)^{\delta_1} \frac{1}{\delta_1!}\cdot  2^{\delta_2}\left(\frac{1}{3}\right)^{\delta_2} \frac{1}{\delta_2!}\cdot
 \left(\frac{3}{2}\right)^{\delta_3}\left(\frac{1}{3}\right)^{\delta_3}\frac{1}{\delta_3!} d^{3\delta}=  \left(\frac{1}{3}\right)^{\delta}
 2^{\delta_2}\cdot \left(\frac{3}{2}\right)^{\delta_3}\cdot \frac{1}{\delta_1!\delta_2!\delta_3!}\cdot d^{3\delta}.$$

We sum the coefficient of this expression over all $(\delta_1,\delta_2,\delta_3)$ and obtain

\begin{align*}
&\left(\frac{1}{3}\right)^{\delta} \sum_{\delta_1=0}^{\delta}\sum_{\delta_2=0}^{\delta-\delta_1} 2^{\delta_2}\left(\frac{3}{2}\right)^{\delta-\delta_1-\delta_2}\cdot \frac{1}{\delta_1!\delta_2!(\delta-\delta_1-\delta_2)!}\\=&
\left(\frac{1}{3}\right)^{\delta} \sum_{\delta_1=0}^{\delta}\frac{1}{\delta_1!(\delta-\delta_1!)}\sum_{\delta_2=0}^{\delta-\delta_1} 2^{\delta_2}\left(\frac{3}{2}\right)^{\delta-\delta_1-\delta_2}\cdot \binom{\delta-\delta_1}{\delta_2}\\=&
\left(\frac{1}{3}\right)^{\delta} \sum_{\delta_1=0}^{\delta}\frac{1}{\delta_1!(\delta-\delta_1!)}\left(2+\frac{3}{2}\right)^{\delta-\delta_1}=
\left(\frac{1}{3}\right)^{\delta} \frac{1}{\delta!}\sum_{\delta_1=0}^{\delta}\binom{\delta}{\delta_1}
\left(\frac{7}{2}\right)^{\delta-\delta_1}\\=&
\left(\frac{1}{3}\right)^{\delta}\frac{1}{\delta!} \left(1+\frac{7}{2}\right)^\delta = \left(\frac{1}{3}\cdot \frac{9}{2}\right)^{\delta}\frac{1}{\delta!}= \left(\frac{3}{2}\right)^{\delta}\frac{1}{\delta!}.
\end{align*}

\end{proof}

This proves Theorem \ref{thm-realsurfacesconstructive}.

\begin{theorem}\label{thm-floorplanasymptoticP2P1}
The number $ N_{\delta,\C}^{\PP^1\times\PP^2,\floor}(d,e)$ of $\delta$-nodal floor plans for surfaces of bidegree $(d,e)$ satisfies $$N_{\delta,\C}^{\PP^1\times\PP^2,\floor}(d,e)= (12de^2)^\delta/\delta!+O(d^{\delta-1}e^{2\delta})+O(d^{\delta}e^{2\delta-1}).$$
\end{theorem}
\begin{proof}
We proceed as in the proof of Theorem \ref{thm-floorplanasymptoticP3}.

Now the number of $C_{i_j}$ satisfying the conditions with a node germ of type (1) is $3e^2+O(e)$, and the complex multiplicity of such a node germ is defined to be twice the multiplicity of the curve (see Definition \ref{def-complexmultP2P1}). Thus, for each $i_j$, $j=1,\ldots,\delta_1$, we obtain a factor of $6e^2$.

The number of $C_{i_j}$ with a node germ of type (2) is $2e$, since we have that many possible locations for horizontal or diagonal ends of weight $2$. The multiplicity for each is $2e$ by Definition \ref{def-complexmultP2P1}, so altogether we obtain a factor of $4e^2$.

Finally, consider a curve $C_{i_j}$ with a left string.
It can align with any horizontal bounded edge of $C_{i_j+1}$, and there are $\frac{1}{2}e^2+O(i_j)$ many. For each, the local complex multiplicity is $2$, and so we have a contribution of $e^2$ for these. If $C_{i_j}$ has a right string, we disregard alignments with a diagonal bounded edge and only consider alignements with a vertex not adjacent to a diagonal edge, of which there are $e^2+O(e)$ many. Altogether, we obtain a factor of $e^2+e^2=2e^2$ for all possibilities for $C_{i_j}$ which have a string, counted with multiplicity.

Thus, the leading term of the contribution of all floor plans with fixed partition and fixed choice of indices as above is
$$ (6e^2)^{\delta_1}\cdot (4e^2)^{\delta_2}\cdot (2e^2)^{\delta_3}.$$

Again, we sum this over all choices of indices, obtaining:
\begin{align*}
&\sum_{i_1=1}^{d} \sum_{i_2=1}^{i_1} \ldots \sum_{i_{\delta_1}=1}^{i_{\delta_1-1}} \sum_{i_{\delta_1+1}=1}^{d} \ldots \sum_{i_{\delta_1+\delta_2}=1}^{i_{\delta_1+\delta_2-1}} \sum_{i_{\delta_1+\delta_2+1}=1}^{d} \ldots \sum_{i_{\delta}=1}^{i_{\delta-1}}
 (6e^2)^{\delta_1}\cdot (4e^2)^{\delta_2}\cdot (2e^2)^{\delta_3} \\ =& 6^{\delta_1}4^{\delta_2}2^{\delta_3}\frac{1}{\delta_1!\delta_2!\delta_3!}(de^2)^\delta.
 \end{align*}

What remains to be done is to sum the coefficient of this expression over all $(\delta_1,\delta_2,\delta_3)$:

\begin{align*}
&\sum_{\delta_1=0}^{\delta}\sum_{\delta_2=0}^{\delta-\delta_1} 6^{\delta_1}4^{\delta_2}2^{\delta-\delta_1-\delta_2} \cdot \frac{1}{\delta_1!\delta_2!(\delta-\delta_1-\delta_2)!} \\&=
\frac{1}{\delta!}\sum_{\delta_1=0}^{\delta}\frac{\delta!}{\delta_1!(\delta-\delta_1)!} 6^{\delta_1} \sum_{\delta_2=0}^{\delta-\delta_1} 4^{\delta_2}2^{\delta-\delta_1-\delta_2} \cdot \frac{(\delta-\delta_1)!}{\delta_1!\delta_2!(\delta-\delta_1-\delta_2)!}
\\=& \frac{1}{\delta!}\sum_{\delta_1=0}^{\delta}\frac{\delta!}{\delta_1!(\delta-\delta_1)!} 6^{\delta_1} (4+2)^{\delta-\delta_1} = \frac{1}{\delta!} (6+6)^\delta = \frac{1}{\delta!}12^\delta.
\end{align*}
\end{proof}
This asymptotics in the complex case was known to the experts. However, now we use our method to consider the real case:

\begin{theorem}\label{thm-realfloorplanasymptoticP2P1}
Let $s=((+)^3)^n$ be the sign vector with all signs positive, and let integers $d,e,\delta$ satisfy $0<\delta\le\frac{1}{2}\min\{d,e\}$.
Then
the number of $\delta$-nodal floor plans of bidegree $(d,e)$, counted with real multiplicity w.r.t. $s$, satisfies:
\begin{equation}N_{\delta,\R,s}^{\PP^1\times\PP^2,\floor}(d,e) \ge\frac{1}{\delta!}\left(\frac{11}{2}de^2\right)^\delta+O(d^{\delta-1}e^{2\delta})
+O(d^\delta e^{2\delta-1}),\label{ne7}\end{equation}
where $\delta$ is fixed and $\min\{d,e\}\to\infty$.
\end{theorem}

\begin{proof}
Suppose, first, that $e\equiv0\mod4$.

Similarly to the proof of Theorem \ref{thm-realfloorplanasymptoticP3positive}, we have to adapt the proof of Theorem \ref{thm-floorplanasymptoticP2P1} to the real multiplicities of Definition \ref{def-realmult1P1P2}.

Now the number of $C_{i_j}$ satisfying the conditions with a node germ of type (1) which is a paralellogram is $e^2+O(e)$, and the real multiplicity of such a node germ is defined to be twice the multiplicity of the curve (see Definition \ref{def-realmult1P1P2}). Thus, for each $i_j$, $j=1,\ldots,\delta_1$, we obtain a factor of $2e^2$.

The number of $C_{i_j}$ with a node germ of type (2) is $2e$, since we have that many possible locations for horizontal or diagonal ends of weight $2$. The multiplicity is $0$ for the horizontal ends and $2e$ for the diagonal ends by Definition \ref{def-realmult1P1P2}, so altogether we obtain a factor of $2e^2$.

Finally, consider a curve $C_{i_j}$ with a left string.
It can align with any horizontal bounded edge of $C_{i_j+1}$, and there are $\frac{1}{2}e^2+O(i_j)$ many, but only half of them give a nonzero contribution of $2$. We obtain $\frac{1}{2}e^2$ for these. If $C_{i_j}$ has a right string, we disregard alignements with a diagonal bounded edge and only consider alignements with a vertex not adjacent to a diagonal edge, of which there are $e^2+O(e)$ many. Altogether, we obtain a factor of $\frac{1}{2}e^2+e^2=\frac{3}{2}e^2$ for all possibilities for $C_{i_j}$ which have a string, counted with multiplicity.

Thus, the leading term of the contribution of all floor plans with fixed partition and fixed choice of indices as above is
$$ (2e^2)^{\delta_1}\cdot (2e^2)^{\delta_2}\cdot \left(\frac{3}{2}e^2\right)^{\delta_3}.$$

Again, we sum this over all choices of indices, obtaining:
\begin{align*}
&\sum_{i_1=1}^{d} \sum_{i_2=1}^{i_1} \ldots \sum_{i_{\delta_1}=1}^{i_{\delta_1-1}} \sum_{i_{\delta_1+1}=1}^{d} \ldots \sum_{i_{\delta_1+\delta_2}=1}^{i_{\delta_1+\delta_2-1}} \sum_{i_{\delta_1+\delta_2+1}=1}^{d} \ldots \sum_{i_{\delta}=1}^{i_{\delta-1}}
 (2e^2)^{\delta_1}\cdot (2e^2)^{\delta_2}\cdot \left(\frac{3}{2}e^2\right)^{\delta_3} \\ =& 2^{\delta_1}2^{\delta_2}\left(\frac{3}{2}\right)^{\delta_3}\frac{1}{\delta_1!\delta_2!\delta_3!}(de^2)^\delta.
 \end{align*}

What remains to be done is to sum the coefficient of this expression over all $(\delta_1,\delta_2,\delta_3)$:

\begin{align*}
&\sum_{\delta_1=0}^{\delta}\sum_{\delta_2=0}^{\delta-\delta_1} 2^{\delta_1}2^{\delta_2}\frac{3}{2}^{\delta-\delta_1-\delta_2} \cdot \frac{1}{\delta_1!\delta_2!(\delta-\delta_1-\delta_2)!} \\&=
\frac{1}{\delta!}\sum_{\delta_1=0}^{\delta}\frac{\delta!}{\delta_1!(\delta-\delta_1)!} 2^{\delta_1} \sum_{\delta_2=0}^{\delta-\delta_1} 2^{\delta_2}\left(\frac{3}{2}\right)^{\delta-\delta_1-\delta_2} \cdot \frac{(\delta-\delta_1)!}{\delta_1!\delta_2!(\delta-\delta_1-\delta_2)!}
\\=& \frac{1}{\delta!}\sum_{\delta_1=0}^{\delta}\frac{\delta!}{\delta_1!(\delta-\delta_1)!} 2^{\delta_1} \left(2+\frac{3}{2}\right)^{\delta-\delta_1} = \frac{1}{\delta!} \left(2+\frac{7}{2}\right)^\delta = \frac{1}{\delta!}\left(\frac{11}{2}\right)^\delta.
\end{align*}

%
%

Suppose that $e\not\equiv0\mod4$. Set $e'=4[e/4]$. Then, for the data $d,e',\delta$, we have
$$N_{\delta,\R,s}^{\PP^1\times\PP^2,\floor}(d,e') \ge\frac{1}{\delta!}\left(\frac{11}{2}d(e')^2\right)^\delta+O(d^{\delta-1}(e')^{2\delta})
+O(d^\delta (e')^{2\delta-1})$$
$$\qquad\qquad\qquad \sim \frac{1}{\delta!}\left(\frac{11}{2}de^2\right)^\delta+O(d^{\delta-1}e^{2\delta})
+O(d^\delta e^{2\delta-1})\ .$$ Thus, the following lemma completes the proof.
\end{proof}

\begin{lemma}\label{nl1}
Let $\Delta'\subset\Delta\subset\R^3$ be nondegenerate convex lattice polytopes, $\delta\ge1$. Suppose that there exists a configuration $\bp'$ of $|\Delta'\cap\Z^3|-1-\delta$ points in $(\R^*)^3$ such that the linear system of surfaces with Newton polytope $\Delta'$ passing through $\bp'$ contains $N$ real surfaces with $\delta$ real nodes as their only singularities (in the big torus), and each of these surfaces corresponds to a point of a transversal intersection of the linear system with the family of $\delta$-nodal surfaces with Newton polytope $\Delta'$. Then one can extend $\bp'$ up to a configuration $\bp$ of
$|\Delta\cap\Z^3|-1-\delta$ points in $(\R^*)^3$ so that the linear system of surfaces with Newton polygon $\Delta$ passing through $\bp$ will contain (at least) $N$ real surfaces with $\delta$ real nodes as their only singularities.
\end{lemma}

\begin{proof}
The statement can be proved literally in the same manner as \cite[Theorem 5.3]{MMS15}.
\end{proof}

\begin{theorem}\label{thm-floorplanasymptoticP1P1P1}
The number $ N_{\delta,\C}^{\PP^1\times\PP^1\times\PP^1,\floor}(d,e,f)$ of $\delta$-nodal floor plans for surfaces of tridegree $(d,e,f)$ satisfies $$N_{\delta,\C}^{\PP^1\times\PP^1\times\PP^1,\floor}(d,e,f)= (24def)^\delta/\delta!+O(d^{\delta-1}e^{\delta}f^\delta)+O(d^{\delta}e^{\delta-1}f^\delta)+O(d^\delta e^\delta f^{\delta-1}).$$
\end{theorem}
\begin{proof}
We proceed as in the proof of Theorem \ref{thm-floorplanasymptoticP3}.

Now the number of $C_{i_j}$ satisfying the conditions with a node germ of type (1) has leading term $6ef$, so we obtain a factor of $12ef$.

The number of $C_{i_j}$ with a node germ of type (2) is $2e+2f$, since we have that many possible locations for horizontal or vertical ends of weight $2$. The multiplicity for each horizontal one is $2f$ by Definition \ref{def-complexmultP3}, whereas for the vertical ones we have $2e$, so altogether we obtain a factor of $8ef$.

Finally, consider a curve $C_{i_j}$ with a left string.
It can align with any horizontal bounded edge of $C_{i_j+1}$, and there are $ef$ many asymptotically. For each, the local complex multiplicity is $2$, and so we have a contribution of $2ef$ for these. The same holds for right strings. Altogether, we have $4ef$.

Thus, the leading term of the contribution of all floor plans with fixed partition and fixed choice of indices as above is
$$ (12ef)^{\delta_1}\cdot (8ef)^{\delta_2}\cdot (4ef)^{\delta_3}.$$

Again, we sum this over all choices of indices, obtaining:
\begin{align*}
&\sum_{i_1=1}^{d} \sum_{i_2=1}^{i_1} \ldots \sum_{i_{\delta_1}=1}^{i_{\delta_1-1}} \sum_{i_{\delta_1+1}=1}^{d} \ldots \sum_{i_{\delta_1+\delta_2}=1}^{i_{\delta_1+\delta_2-1}} \sum_{i_{\delta_1+\delta_2+1}=1}^{d} \ldots \sum_{i_{\delta}=1}^{i_{\delta-1}}
 (12ef)^{\delta_1}\cdot (8ef)^{\delta_2}\cdot (4ef)^{\delta_3}\\
 = &\frac{1}{\delta_1!\delta_2!\delta_3!}12^{\delta_1}8^{\delta_2}4^{\delta_3} (def)^\delta.
 \end{align*}

What remains to be done is to sum the coefficient of this expression over all $(\delta_1,\delta_2,\delta_3)$, and in the same way as before we obtain $\frac{1}{\delta!}\cdot 24^\delta$.
\end{proof}

Again, this asymptotics in the complex case was known to the experts. As before, we now use our methods to consider the real case:

\begin{theorem}\label{thm-realfloorplanasymptoticP1P1P1}
Let integers $d,e,f,\delta$ satisfy $0<\delta<\frac{1}{2}\min\{d,e,f\}$, and let the sign vector $s$ be as introduced in
Definition \ref{signvectorP1P1P1}. Then
the number of $\delta$-nodal floor plans of tridegree $(d,e,f)$, counted with real multiplicity w.r.t. $s$, satisfies:
\begin{equation}N_{\delta,\R,s}^{\PP^1\times\PP^1\times\PP^1,\floor}(d,e,f) =\frac{(10def)^\delta}{\delta!}+O(d^{\delta-1}e^{\delta}f^\delta)+O(d^{\delta}e^{\delta-1}f^\delta)+O(d^\delta e^\delta f^{\delta-1}),\label{ne10}\end{equation} where $\delta$ is fixed and $\min\{d,e,f\}\to\infty$.
\end{theorem}

\begin{proof}
Suppose, first, that $e$ is even. We adapt Theorem \ref{thm-floorplanasymptoticP1P1P1} to the real multiplicities of Definition \ref{def-realmult1P1P1P1}.

The number of $C_{i_j}$ satisfying the conditions with a node germ of type (1) has leading term $6ef$, but only $2ef$ of those come from parallelograms and the others from midpoints of weight $2$ edges which do not contribute, so we obtain a factor of $4ef$.

The number of $C_{i_j}$ with a node germ of type (2) is $2e+2f$, since we have that many possible locations for horizontal or vertical ends of weight $2$. But only the vertical ones contribute with multiplicity $2f$, so altogether we obtain a factor of $4ef$.

Finally, consider a curve $C_{i_j}$ with a left string.
It can align with any horizontal bounded edge of $C_{i_j+1}$, and there are $ef$ many asymptotically, but only half of them contribute. For each, the local complex multiplicity is $2$, and so we have a contribution of $ef$ for these. The same holds for right strings. Altogether, we have $2ef$.

Thus, the leading term of the contribution of all floor plans with fixed partition and fixed choice of indices as above is
$$ (4ef)^{\delta_1}\cdot (4ef)^{\delta_2}\cdot (2ef)^{\delta_3}.$$

Again, we sum this over all choices of indices, obtaining:
\begin{align*}
&\sum_{i_1=1}^{d} \sum_{i_2=1}^{i_1} \ldots \sum_{i_{\delta_1}=1}^{i_{\delta_1-1}} \sum_{i_{\delta_1+1}=1}^{d} \ldots \sum_{i_{\delta_1+\delta_2}=1}^{i_{\delta_1+\delta_2-1}} \sum_{i_{\delta_1+\delta_2+1}=1}^{d} \ldots \sum_{i_{\delta}=1}^{i_{\delta-1}}
 (4ef)^{\delta_1}\cdot (4ef)^{\delta_2}\cdot (2ef)^{\delta_3}\\
 = &\frac{1}{\delta_1!\delta_2!\delta_3!}4^{\delta_1}4^{\delta_2}2^{\delta_3} (def)^\delta.
 \end{align*}

What remains to be done is to sum the coefficient of this expression over all $(\delta_1,\delta_2,\delta_3)$, and in the same way as before we obtain $\frac{1}{\delta!}\cdot 10^\delta$.

Suppose that $e$ is odd. Then we prove (\ref{ne10}) by means of Lemma \ref{nl1} as it was done in the proof of Theorem \ref{thm-realfloorplanasymptoticP2P1}.
\end{proof}

\bibliographystyle{plain}
\bibliography{bibliographie}

\end {document}